\documentclass{amsart}
\usepackage{amsfonts,amscd,amsthm,amsgen,amsmath,amssymb}
\usepackage[all]{xy}
\usepackage{epsfig,color}
\usepackage[vcentermath]{youngtab}
\newtheorem{theorem}{Theorem}[section]
\newtheorem{lemma}[theorem]{Lemma}

\newtheorem{proposition}[theorem]{Proposition}

\theoremstyle{definition}
\newtheorem{definition}[theorem]{Definition}

\theoremstyle{remark}
\newtheorem{remark}[theorem]{Remark}

\numberwithin{equation}{section}

\begin{document}

\title{Non-trivial smooth families of $K3$ surfaces}
\author{David Baraglia}

\address{School of Mathematical Sciences, The University of Adelaide, Adelaide SA 5005, Australia}

\email{david.baraglia@adelaide.edu.au}



\date{\today}


\begin{abstract}
Let $X$ be a complex $K3$ surface, ${\rm Diff}(X)$ the group of diffeomorphisms of $X$ and ${\rm Diff}_0(X)$ the identity component. We prove that the fundamental group of ${\rm Diff}_0(X)$ contains a free abelian group of countably infinite rank as a direct summand. The summand is detected using families Seiberg--Witten invariants. The moduli space of Einstein metrics on $X$ is used as a key ingredient in the proof.
\end{abstract}

\maketitle


\section{Introduction}

There is considerable interest in understanding the topology of diffeomorphism groups of $4$-manifolds. While much remains unknown there has been some recent progress.

\begin{itemize}
\item{Ruberman gave examples of simply-connected smooth $4$-manifolds for which $\pi_0({\rm Diff}(X)) \to \pi_0({\rm Homeo}(X))$ is not injective \cite{rub1,rub2}.}
\item{Watanabe constructed many non-trivial homotopy classes in ${\rm Diff}(S^4)$, thereby disproving the $4$-dimensional Smale conjecture \cite{wat}.}
\item{Baraglia--Konno showed that $\pi_1({\rm Diff}(K3)) \to \pi_1({\rm Homeo}(K3))$ is not surjective \cite{bako3}.}
\item{Smirnov showed that if $X$ is a hypersurface in $\mathbb{CP}^3$ of degree $d \neq 1,4$, then $\pi_1({\rm Diff}(X))$ is non-trivial. Note that this result excludes $K3$, which corresponds to hypersurfaces of degree $d=4$ \cite{sm}.}
\end{itemize}

The main result of this paper is the following.

\begin{theorem}
Let $X$ be a $K3$ surface. Then $\pi_1({\rm Diff}_0(X))$ contains a free abelian group of countably infinite rank as a direct summand.
\end{theorem}

In particular, $\pi_1({\rm Diff}_0(X))$ is not finitely generated. This contrasts with a recent theorem of Bustamante--Krannich--Kupers \cite{bkk} who showed that if $M$ is a closed smooth manifold of dimension $2n \ge 6$ with finite fundamental group, then the homotopy groups of ${\rm Diff}(M)$ are finitely generated.

The direct summand in the above theorem is detected using families Seiberg--Witten invariants. The families Seiberg--Witten invariants were originally defined in \cite{liliu}. In this paper we consider a reformulation of the Seiberg--Witten invariants which we now outline. Let $X$ be a compact, oriented smooth $4$-manifold with $b^+(X) > 1$. Let $\mathfrak{s}$ be a spin$^c$-structure on $X$ and let
\[
d(X , \mathfrak{s}) = \frac{ c_1(\mathfrak{s})^2 - \sigma(X) }{4} - 1 +b_1(X) - b^+(X)
\]
be the expected dimension of the Seiberg--Witten moduli space, where $\sigma(X)$ is the signature of $X$. If $d(X, \mathfrak{s}) \le -2$, then we construct a map
\[
sw_{\mathfrak{s}} : \pi_{-d(X,\mathfrak{s})-1}( {\rm Diff}_0(X) ) \to \mathbb{Z}.
\]
The definition, roughly, is as follows. Let $f \in \pi_{-d(X,\mathfrak{s}) -1}({\rm Diff}_0(X))$. Then by the clutching construction, $f$ defines a family $E_f \to S^{-d(X,\mathfrak{s})}$ over the sphere with fibres diffeomorphic to $X$. The moduli space of solutions to the Seiberg--Witten equations on the family $E_f$ with spin$^c$-structure $\mathfrak{s}$ is compact and has expected dimension $d(X,\mathfrak{s}) - d(X,\mathfrak{s}) = 0$. For a generic perturbation the families moduli space is a compact oriented $0$-manifold and $sw_{\mathfrak{s}}(f)$ is defined as a signed count of the points of this moduli space. If $b^+(X) \le -d(X,\mathfrak{s})+1$, then one has to deal with wall crossing phenomena. However, we show that in the above situation there is a canonically defined chamber and we take $sw_{\mathfrak{s}}(f)$ to be the Seiberg--Witten invariant defined with respect to this chamber. This subtlety is crucial to this paper, since we will be concerned with the case $b^+(X) = 3$ and $d(X,\mathfrak{s}) = -2$.

We prove a number of results concerning the invariants $sw_{\mathfrak{s}}$.

\begin{theorem}
Let $X$ be a compact, oriented, smooth $4$-manifold with $b^+(X) > 1$. Then for each spin$^c$-structure with $d(X,\mathfrak{s}) = -(n+1) \le -2$, the map 
\[
sw_{\mathfrak{s}} : \pi_n({\rm Diff}_0(X)) \to \mathbb{Z}
\]
is a group homomorphism.
\end{theorem}

\begin{theorem}\label{thm:f}
Assume that $b_1(X)=0$. Then for any given $f \in \pi_n({\rm Diff}_0(X))$, $sw_{\mathfrak{s}}(f)$ is non-zero for only finitely many spin$^c$-structures with $d(X,\mathfrak{s}) = -(n+1)$.
\end{theorem}

Theorem \ref{thm:f} is essentially a consequence of the compactness properties of the Seiberg--Witten equations. However there is a subtlety due to the chamber structure and wall crossing that requires some non-trivial arguments to overcome. From these two theorems it follows that (for each $n \ge 1$) we can put the Seiberg--Witten invariants together into a single homomorphism
\[
sw : \pi_n({\rm Diff}_0(X)) \to \bigoplus_{\mathfrak{s} | d(X,\mathfrak{s}) = -(n+1) } \! \! \! \! \! \! \! \! \! \! \! \! \mathbb{Z}, \quad x \mapsto \bigoplus_{\mathfrak{s}} sw_{\mathfrak{s}}(x).
\]

Now let $X$ be a $K3$ surface. Let
\[
\Delta = \{ \alpha \in H^2(X ; \mathbb{Z}) \; | \; \alpha^2 = -2 \}
\]
be the ``roots" of $X$. For each $\alpha \in \Delta$ we get a unique spin$^c$-structure $\mathfrak{s}_{\alpha}$ characterised by $c_1(\mathfrak{s}_\alpha) = 2\alpha$ (since $X$ is simply-connected and spin, the map $\mathfrak{s} \to c_1(\mathfrak{s})$ is a bijection between spin$^c$-structures and elements of $H^2(X ; \mathbb{Z})$ that are divisible by $2$). Then $d(X , \mathfrak{s}_{\alpha}) = -2$ and so we have a homomorphism $sw_{\alpha} : \pi_1({\rm Diff}(X)) \to \mathbb{Z}$.

Choose an element $v \in H^2(X ; \mathbb{R})$ such that $\langle v , \delta \rangle \neq 0$ for all $\delta \in \Delta$ and define $\Delta^{\pm} = \{ \delta \in \Delta \; | \; \pm \langle v , \delta \rangle > 0 \}$. Then
\[
\Delta = \Delta^+ \cup \Delta^-
\]
and $\delta \in \Delta^+$ if and only if $-\delta \in \Delta^-$. The reason for splitting up $\Delta$ this way is that the invariants $sw_{\mathfrak{\alpha}}$ and $sw_{\mathfrak{-\alpha}}$ are related to one another by the charge conjugation symmetry of the Seiberg--Witten equations. In fact, $sw_{\mathfrak{\alpha}} = -sw_{\mathfrak{-\alpha}}$ (see Proposition \ref{prop:conj}).

In \textsection \ref{sec:ein} we recall the construction of the moduli space $T_{Ein}$ of Einstein metrics on $X$, which may be regarded as an analogue of Teichm\"uller space for $K3$ surfaces. Over $T_{Ein}$ is a universal family $E_{Ein} \to T_{Ein}$. For each $\delta \in \Delta^+$, we construct a homotopy class of map $g_\delta : S^2 \to T_{Ein}$. Let $E_\delta \to S^2$ be the family over $S^2$ obtained by pulling back the universal family under $g_\delta$. Using the geometry of $T_{Ein}$, we compute the families Seiberg--Witten invariant of $E_\delta$. This gives the following result.

\begin{theorem}
Let $\alpha, \delta \in \Delta^+$. Then
\[
sw_{\alpha}( h_\delta ) = \begin{cases} 1 & \text{if } \alpha = \delta, \\ 0 & \text{otherwise}. \end{cases}
\]
\end{theorem}

Our main theorem follows directly from this. Note that $\pi_1( {\rm Diff}_0(X) )$ is abelian since ${\rm Diff}_0(X)$ is a topological group.

A brief outline of the paper is as follows. In \textsection \ref{sec:fsw} we recall the construction of the families Seiberg--Witten invariants. We then show the invariants can be reformulated as maps $sw_{\mathfrak{s}} : \pi_n({\rm Diff}_0(X)) \to \mathbb{Z}$ and prove several properties of these invariants, in particular Theorems \ref{thm:hom} and \ref{thm:finite}. In \textsection \ref{sec:ein} we specialise to the case that $X$ is a $K3$ surface. We construct the family $E_{Ein} \to T_{Ein}$ over the Teichm\"uller space $T_{Ein}$ and use this to contruct classes $h_\delta \in \pi_1({\rm Diff}_0(X))$. We then compute the Seiberg--Witten invariants of these classes and our main theorem follows.

\section{The families Seiberg--Witten invariant revisited}\label{sec:fsw}

In this section we will recall the definition of the families Seiberg--Witten invariant. We will also show that the definition of the invariant can be extended to situations where wall-crossing phenomena are present. We show that under certain conditions a distinguished chamber exists, hence we can still obtain a well-defined invariant.

Our approach to the families Seiberg--Witten invariant follows \cite{liliu} but with some additional modifications as in \cite{bako1}. Let $X$ be a compact smooth oriented $4$-manifold and let $B$ be a compact smooth manifold. Suppose we have a smooth fibrewise oriented fibre bundle $\pi : E \to B$ whose fibres are diffeomorphic to $X$. Such a fibre bundle will be called a {\em smooth family over $B$ with fibres diffeomorphic to $X$}. We assume throughout that $B$ is connected. Choose a basepoint $p \in B$ and a diffeomorphism $X_p \cong X$, where $X_p = \pi^{-1}(p)$ denotes the fibre of $E$ over $p$. Then $\pi_1(B , p)$ acts by monodromy on the set of spin$^c$-structures on $X$. Suppose that $\mathfrak{s}$ is a monodromy invariant spin$^c$-structure on $X$. Then by monodromy invariance, $\mathfrak{s}$ can be uniquely extended to a continuously varying family of spin$^c$-structures $\tilde{\mathfrak{s}} = \{ \mathfrak{s}_b \}_{b \in B}$ on the fibres of $E$ such that $\tilde{\mathfrak{s}}|_{X_p} \cong \mathfrak{s}$ (here continuously varying means that the family admits local trivialisations for which the spin$^c$-structure is constant). Note that the existence of the continuous family $\tilde{\mathfrak{s}}$ is in general a {\em weaker} condition than requiring the existence of a spin$^c$-structure on the vertical tangent bundle $T(E/B) = Ker(\pi_*)$ (because a spin$^c$-structure on the vertical tangent bundle determines a continuously varying family of spin$^c$-structures by taking the fibrewise restriction, but not every continuously varying family of spin$^c$-structures arises this way). However, as explained in \cite{bar1}, \cite{bako1}, the existence of $\tilde{\mathfrak{s}}$ is sufficient to contruct a families Seiberg--Witten moduli space. 

Let
\[
d(X,\mathfrak{s}) = \frac{ c_1(\mathfrak{s})^2 - \sigma(X)}{4} -1 - b^+(X) + b_1(X)
\]
be the virtual dimension of the ordinary Seiberg--Witten moduli space of $X$. Let $g = \{ g_b \}_{b \in B}$ be a smoothly varying family of metrics on the fibres of $E$. Equivalently, $g$ is a metric on the vertical tangent bundle $T(E/B)$. Then we define $\mathcal{H}^+_g(X)$ to be the vector bundle on $B$ whose fibre over $b \in B$ is the space $H^+_{g_b}(X_b)$ of $g_b$-self-dual harmonic $2$-forms. By a families perturbation $\eta$ we mean a smoothly varying family $\eta = \{ \eta_b\}_{b \in B}$ of real $2$-forms on the fibres of $E$, such that $\eta_b$ is $g_b$-self-dual. Let $[\eta_b] \in H^+_{g_b}(X_b)$ denote the $L^2$-othogonal projection of $\eta_b$ to the space of self-dual harmonic forms (using the $L^2$-metric defined by $g_b$). The map $b \mapsto [\eta_b]$ defines a section of $\mathcal{H}^+_g(X)$, which we denote by $[\eta]$.

Recall that the Seiberg--Witten equations for $(X_b , \mathfrak{s}_b , g_b)$ with perturbation $\eta_b$ are:
\begin{align*}
D_A \psi &= 0, \\
F^+_A + i \eta_b &= \sigma(\psi),
\end{align*}
where $A$ is a spin$^c$-connection, $\psi$ is a positive spinor for the spin$^c$-structure $\mathfrak{s}_b$ and $\sigma(\psi)$ denotes the imaginary self-dual $2$-form corresponding to the trace-free part of $\psi^* \otimes \psi$ under Clifford multiplication. Let $w : B \to \mathcal{H}^+_g(X)$ be the section of $\mathcal{H}^+_g(X)$ sending $b$ to $2\pi ( c_1(\mathfrak{s}_b) )^{+_{g_b}}$, the orthogonal projection of $2 \pi c_1(\mathfrak{s}_b)$ to $H^+_{g_b}(X_b)$ using the $L^2$-metric defined by $g_b$. Then the $\eta_b$-perturbed Seiberg--Witten equations for $(X_b , \mathfrak{s}_b , g_b)$ admits reducible solutions if and only if $[\eta_b] = w$ (recall that a solution $(A , \psi)$ of the Seiberg--Witten equations is called reducible if $\psi= 0$ \cite{nic}). We refer to $w$ as the ``wall" and we say that the families perturbation $\eta$ does not lie on the wall if for all $b \in B$, we have $[\eta_b] \neq w_b$.

We define a {\em chamber} (of the families Seiberg--Witten equations) for $(E , \mathfrak{s})$ to be a connected component of the space of pairs $( g , \eta)$, where $g$ is a family of metrics and $\eta$ is a family of perturbations not lying on the wall. In general, there are obstructions to the existence of chambers. For instance, if $b^+(X) = 0$, then there does not exist a chamber. On the other hand, if $b^+(X) > {\rm dim}(B) + 1$, then there exists a unique chamber \cite{bako2}.

Let $\mathcal{C}$ be a chamber of $(E , \mathfrak{s} )$. Then for a sufficiently generic element $( g , \eta ) \in \mathcal{C}$, the moduli space $\mathcal{M}( E , \mathfrak{s} , g , \eta)$ of gauge equivalence classes of solutions to the Seiberg--Witten equations on the fibres of $E$ (with respect to the spin$^c$-structure $\tilde{\mathfrak{s}}$, metric $g$ and perturbation $\eta$) is a smooth, compact manifold of dimension $d(X,\mathfrak{s}_X) + {\rm dim}(B)$ (or is empty if this number is negative). See \cite{liliu} for more details concerning the construction of the families Seiberg--Witten moduli space. Recall that a homology orientation for $X$ is a choice of orientation of $H^+(X) \oplus H^1(X ; \mathbb{R})$ (here $H^+(X)$ denotes the space of harmonic self-dual $2$-forms with respect to some metric on $X$. It is straightforward to see that the homology orientations for different choices of metrics can be canonically identified with one another). We say that a homology orientation is monodromy invariant if it extends to a continuously varying orientation on the family $\{ H^+_{g_b}(X_b) \oplus H^1(X_b ; \mathbb{R}) \}_{b \in B}$.

Let $\pi : \mathcal{M}(E , \mathfrak{s} , g , \eta) \to B$ be the projection to $B$. A monodromy invariant homology orientation defines an orientation on $T_{\mathcal{M}(E , \mathfrak{s} , g , \eta)} \oplus \pi^*(TB)$, hence a Gysin homomorphism 
\[
\pi_* : H^j( \mathcal{M}(E , \mathfrak{s} , g , \eta) ; \mathbb{Z} ) \to H^{j-d(X , \mathfrak{s})}( B ; \mathbb{Z}).
\]

We define the {\em families Seberg--Witten invariant} $SW(E , \mathfrak{s} , \mathcal{C} , \mathfrak{o}) \in H^{-d(X,\mathfrak{s})}(B ; \mathbb{Z})$ of $E$ with respect to the monodromy invariant spin$^c$-structure $\mathfrak{s}$, the chamber $\mathcal{C}$ and monodromy invariant homology orientation $\mathfrak{o}$ to be
\[
SW(E , \mathfrak{s} , \mathcal{C} , \mathfrak{o} ) = \pi_*(1) \in H^{-d(X,\mathfrak{s})}(B ; \mathbb{Z}).
\]
The fact that $SW(E , \mathfrak{s} , \mathcal{C} , \mathfrak{o} )$ depends only on the chamber $\mathcal{C}$ and not on the particular choice of pair $(g , \eta) \in \mathcal{C}$ follows by much the same argument as in the unparametrised case. A generic path between pairs $(g,\eta), (g' , \eta')$ determines a cobordism (relative $B$) of the moduli spaces $\mathcal{M}(E , \mathfrak{s} , g , \eta)$ and $ \mathcal{M}(E , \mathfrak{s} , g' , \eta')$.

Let $\pi : E \to B$ be a smooth family over $B$ with fibres diffeomorphic to $X$. Let $\mathcal{H}^2(X)$ denote the local system over $B$ whose fibre over $b \in B$ is $\mathcal{H}^2(X)_b = H^2(X_b ; \mathbb{R})$. Assume that $B$ is simply-connected and that $b^+(X) > 1$. Choose a basepoint $p \in B$. Then parallel translation defines a trivialisation
\[
\tau : \mathcal{H}^2(X) \to B \times H^2(X_p ; \mathbb{R}).
\]
Let $H \subseteq H^2(X_p ; \mathbb{R})$ be a maximal positive definite subspace with respect to the intersection form and let $H^\perp$ denote the orthogonal complement of $H$ (with respect to the intersection form). This defines a decomposition $H^2(X_p ; \mathbb{R}) \cong H \oplus H^\perp$. Let $\rho_H : H^2(X_p ; \mathbb{R}) \to H$ denote the projection to the first factor.

Choose a smoothly varying family of metrics $g = \{g_b\}$ and let $\mathcal{H}^+_g(X)$ be defined as before. Let $\iota_g : \mathcal{H}^+_g(X) \to \mathcal{H}^2(X)$ be the inclusion. Let $pr_B : B \times H^2(X_p ; \mathbb{R}) \to B$ be the projection to $B$. Then the composition 
\[
\varphi_g = (pr_B \times \rho_H) \circ \tau \circ \iota_g : \mathcal{H}^+_g(X) \to B \times H
\]
is an isomorphism of vector bundles. This follows since $\tau( \iota_g( \mathcal{H}^+_g(X)))$ is a positive definite subbundle of $B \times H^2(X_p ; \mathbb{R})$, so meets the negative definite subbundle $B \times H^\perp$ in the zero section.

Let $\mathfrak{s}$ be a spin$^c$-structure on $X$. Since $B$ is simply-connected, $\mathfrak{s}$ is automatically monodromy invariant and so continuously extends to the fibres of $E$.

Let $w : B \to \mathcal{H}^+_g(X)$ be the wall with respect to the spin$^c$-structure $\mathfrak{s}$. Then $\varphi_g(w)$ is a section of the trivial bundle $B \times H$. Let 
\[
R_{g} = \sup_{b \in B} || \varphi_g( w_b ) ||_H
\]
where $|| \; ||_H$ is the norm on $H$ induced by the restriction to $H$ of the intersection form on $H^2(X_p ; \mathbb{R})$. Since $B$ is compact, $R_g$ is finite. Since $b^+(X) > 1$, $H$ is a non-zero vector space and hence there exist elements of arbitrarily large norm. Let $v$ be any element of $H$ with $||v||_H > R_g$. Then the constant section $b \mapsto b \times v$ is disjoint from $\varphi_g(w)$. Therefore, the section $v_g = \varphi_g^{-1}(v)$ of $\mathcal{H}^+_g(X)$ is disjoint from the wall $w$ and hence defines a chamber, depending only on $g$ and $v$ which we will denote by $\mathcal{C}(g,v)$.

\begin{lemma}\label{lem:chamb}
The chamber $\mathcal{C}(g,v)$ does not depend on the choice of the pair $(g,v)$.
\end{lemma}
\begin{proof}
First we show that for fixed $g$, the chamber $\mathcal{C}(g,v)$ does not depend on the choice of $v$. Let $R_{g} = \sup_{b \in B} || \varphi_g( w_b ) ||_H$ be defined as before and let $v,v'$ be any two elements of $H$ with $||v||_H, ||v'||_H > R_g$. The space $\{ x \in H \; | \; ||x||_H > R_g \}$ is homotopy equivalent to a sphere of dimension $b^+(X)-1$ and is therefore connected, since we are assuming that $b^+(X) > 1$. Therefore we can find a continuous path $\{v_t\}_{t \in [0,1]}$ in $\{ x \in H \; | \; ||x||_H > R_g \}$ joining $v$ to $v'$. It follows that $(g, (v_t)_g ) \in \mathcal{C}(g,v)$ for all $t \in [0,1]$. Hence $(g,(v')_g) \in \mathcal{C}(g,v)$ and $\mathcal{C}(g,v) = \mathcal{C}(g,v')$.

Now let $g,g'$ be two different families of metrics. We will show that there exists a $v \in H$ for which $\mathcal{C}(g,v) = \mathcal{C}(g',v)$. Together with the above shown independence of $\mathcal{C}(g,v)$ on $v$, this will show that $\mathcal{C}(g,v)$ does not depend on the choice of pair $(g,v)$.

Choose a continuous path $\{g_t\}_{t \in [0,1]}$ of families of metrics from $g$ to $g'$. Let 
\[
R = \sup_{b \in B, t \in [0,1]} || \varphi_{g_t}( w_b ) ||_H.
\] 
Compactness of $B \times [0,1]$ implies that $R$ is finite. Now choose $v \in H$ such that $||v||_H > R$. Then $|| v ||_H > \sup_{b \in B} || \varphi_{g_t}(w_b) ||_H$ for each $t \in [0,1]$. It follows that the pair $(g_t ,  v)$ defines the same chamber for all $t \in [0,1]$. Hence $\mathcal{C}(g,v) = \mathcal{C}(g',v)$.
\end{proof}

\begin{definition}
Let $\pi : E \to B$ be a smooth family over $B$ with fibres diffeomorphic to $X$ and let $\mathfrak{s}$ be a spin$^c$-structure on $X$. Assume that $B$ is simply-connected and that $b^+(X) > 1$. For any pair $(g,v)$ with $v \in \{ x \in H \; | \; ||x||_H > R_g \}$, we let $\mathcal{C}_0(E,\mathfrak{s})$ denote the chamber containing $(g,v_g)$. By Lemma \ref{lem:chamb}, we see that $\mathcal{C}_0(E,\mathfrak{s})$ does not depend on the choice of the pair $(g,v)$. Furthermore, it is clear that $\mathcal{C}_0(E,\mathfrak{s})$ does not depend on the choice of maximal positive definite subspace $H \subseteq H^2(X ; \mathbb{R})$, because the space of all such subspaces is connected (as it can be identified with the connected homogeneous space $O(3,19)/O(3) \times O(19)$). We call $\mathcal{C}_0(E,\mathfrak{s})$ the {\em canonical chamber} for $(E , \mathfrak{s})$.
\end{definition}

\begin{definition}
Let $\pi : E \to B$ be a smooth family over $B$ with fibres diffeomorphic to $X$. Assume that $B$ is simply-connected and that $b^+(X) > 1$. Let $\mathfrak{s}$ be a spin$^c$-structure on $X$ and let $\mathfrak{o}$ be a homology orientation for $X$. Then we define the {\em (canonical) families Seiberg--Witten invariant} $SW(E , \mathfrak{s} , \mathfrak{o})$ of $(E , \mathfrak{s} , \mathfrak{o})$ to be the families Seiberg--Witten invariant of $(E , \mathfrak{s} , \mathfrak{o})$ defined using the canonical chamber:
\[
SW( E , \mathfrak{s} , \mathfrak{o} ) = SW( E , \mathfrak{s} , \mathcal{C}_0(E,\mathfrak{s}) , \mathfrak{o} ) \in H^{-d(X,\mathfrak{s})}(B ; \mathbb{Z}).
\]
\end{definition}

Let $X$ be a compact, oriented, smooth $4$-manifold with $b^+(X) > 1$. Let ${\rm Diff}(X)$ be the group of orientation preserving diffeomorphisms of $X$ with the $\mathcal{C}^\infty$-topology and let ${\rm Diff}_0(X)$ be the identity co
mponent. Let $n >0$ be a positive integer. Consider an element $f \in \pi_n( {\rm Diff}_0(X) )$. Using the clutching construction, $f$ defines a topological fibre bundle $E_f \to S^{n+1}$ over $S^{n+1}$ with fibres homeomorphic to $X$ and structure group ${\rm Diff}_0(X)$. From the main theorem of \cite{mw}, it follows that $E_f$ can be made into a smooth fibre bundle with fibres diffeomorphic to $X$ in a unique way. Hence $E_f$ only depends on the homotopy class of $f$ up to isomorphism as a smooth fibre bundle with structure group ${\rm Diff}_0(X)$.

We describe the clutching construction in detail in order to fix certain orientation conventions. Regard $S^{n+1}$ as the unit sphere in $\mathbb{R}^{n+2}$. The standard orientation on $\mathbb{R}^{n+2}$ induces an orientation on $S^{n+1}$ by the outer normal first convention. Let $S^{n+1}_{\pm} = \{ (x^1 , \dots , x^{n+2}) \in S^{n+1}  | \, \pm x^{n+2} > 0 \}$ be the two hemispheres. Then $S^{n+1} = S^{n+1}_+ \cup S^{n+1}_-$ and $S^{n+1}_+ \cap S^{n+1}_- = S^n$. Given a map $f : S^n \to {\rm Diff}_0(X)$, the fibre bundle $E_f \to S^{n+1}$ is given by taking the trivial bundles $S^{n+1}_+ \times X$, $S^{n+1}_- \times X$ and identifying $(s , x ) \in S^n \times X \subset S^{n+1}_- \times X$ with $(s ,(f(s))(x) ) \in S^n \times X \subset S^{n+1}_+ \times X$.

Let $\mathfrak{s}$ be a spin$^c$-structure on $X$ and $\mathfrak{o}$ a homology orientation. Since $n > 0$, $S^{n+1}$ is simply-connected and $b^+(X) > 1$, the families Seiberg--Witten invariant
\[
SW( E_f , \mathfrak{s} , \mathfrak{o} ) \in H^{-d(X,\mathfrak{s})}( S^{n+1} ; \mathbb{Z})
\]
is defined. If $d(X,\mathfrak{s}) = -(n+1)$, then we can evaluate $SW(E_f , \mathfrak{s} , \mathfrak{o})$ against the fundamental class of $S^{n+1}$ to obtain an integer invariant.

\begin{definition}
Let $X$ be a compact, oriented, smooth $4$-manifold with $b^+(X) > 1$. Let $\mathfrak{s}$ be a spin$^c$-structure such that $d(X,\mathfrak{s}) = -(n+1)$ for some $n \ge 0$. We define
\[
sw_{\mathfrak{s}} : \pi_n( {\rm Diff}_0(X) ) \to \mathbb{Z}
\]
by setting
\[
sw_{\mathfrak{s}}( f ) = \int_{S^{n+1}} SW(E_f , \mathfrak{s} , \mathfrak{o}),
\]
where we have chosen a homology orientation $\mathfrak{o}$ and we have oriented $S^{n+1}$ according to the convention described above. We have omitted from our notation the dependence of $sw_{\mathfrak{s}}$ on the choice of homology orientation. Changing the homology orientation has the effect of changing $sw_{\mathfrak{s}}$ by an overall sign.
\end{definition}

\begin{remark}
The invariant $sw_{\mathfrak{s}}(f) \in \mathbb{Z}$ can be interpreted as follows. Choose a generic pair $(g,\eta) \in \mathcal{C}_0(E_f , \mathfrak{s})$. Then the moduli space $\mathcal{M}(E , \mathfrak{s} , g , \eta)$ is a compact, oriented $0$-manifold and $sw_{\mathfrak{s}}(f)$ is simply the number of points of $\mathcal{M}(E_f , \mathfrak{s} , g , \eta)$, counted with sign.
\end{remark}

\begin{theorem}\label{thm:hom}
Let $X$ be a compact, oriented, smooth $4$-manifold with $b^+(X) > 1$. Then for each spin$^c$-structure with $d(X,\mathfrak{s}) = -(n+1) \le -2$, the map 
\[
sw_{\mathfrak{s}} : \pi_n({\rm Diff}_0(X)) \to \mathbb{Z}
\]
is a group homomorphism.
\end{theorem}
\begin{proof}
Let $f \in \pi_n({\rm Diff}_0(X))$ and let $E_f \to S^{n+1}$ be the corresponding family built from the clutching construction. Choose a generic pair $(g , \eta) \in \mathcal{C}_0(E_f , \mathfrak{s})$. The moduli space $\mathcal{M}(E , \mathfrak{s} , g , \eta)$ is a finite set of points. Let $p_1 , \dots , p_m \in S^{n+1}$ be the finitely many points over which $\mathcal{M}(E_f , \mathfrak{s} , g , \eta)$ lies. Choose a point $p \in S^{n+1}$ and an open disc $D \subset S^{n+1}$ around $p$ such that $D$ is disjoint from $p_1, \dots , p_m$. Furthermore, we can assume that $D$ is contained in the interior of $S^{n+1}_-$. Using cutoff functions it is possible to construct a smooth map $\psi : S^{n+1} \to S^{n+1}$ such that $\psi(D) = \{ p \}$ and $\psi : S^{n+1} \setminus D \to S^{n+1} \setminus \{p\}$ is a diffeomorphism. Now consider the pullback $\psi^*(E_f)$ of $E_f$ under $\psi$. The conditions on $\psi$ ensures that it has degree $1$ as a map of $S^{n+1}$ to itself. Hence $\psi$ is homotopic to the identity. It follows that $\psi^*(E_f)$ is isomorphic to $E_f$ as topological fibre bundles over $S^{n+1}$ with structure group ${\rm Diff}(X)$. The main theorem of \cite{mw} then implies that $\psi^*(E_f) = E_{f \circ \psi}$ and $E_f$ are isomorphic as smooth fibre bundles. The pullback $(\psi^*(g) , \psi^*(\eta))$ is a generic pair in $\mathcal{C}_0(\psi^*(E_f) , \mathfrak{s})$ and the moduli space $\mathcal{M}( \psi^*(E_f) , \mathfrak{s} , \psi^*(g) , \psi^*(\eta) )$ is obviously obtained by pulling back $\mathcal{M}(E_f , \mathfrak{s} , g , \eta)$ by $\psi$. Let $X_p$ denote the fibre of $E_f$ over $p$ and fix a diffeomorphism $X_p \cong X$. Since $\psi$ takes the constant value $p$ on $D$, the restriction of $\psi^*(E_f)$ to $D$ is the constant family $\psi^*(E_f)|_D \cong D \times X_p \cong D \times X$. Under this trivialisation of $\psi^*(E_f)|_D$ we have that $\psi^*(g) , \psi^*(\eta)$ get sent to the constant pair $(g_p , \eta_p)$.

Now let $f'$ be another element of $\pi_n({\rm Diff}_0(X))$ and let $E_{f'} \to S^{n+1}$ be the family corresponding to $f'$. Choose a generic pair $(g',\eta') \in \mathcal{C}_0( E_{f'} , \mathfrak{s})$. Let $r : S^{n+1} \to S^{n+1}$ be the orientation reversing map given by $r(x^1 , \dots , x^{n+2}) = (x^1, \dots , x^{n+1} , -x^{n+2})$. Observe that $r$ exchanges the two hemispheres $S^{n+1}_{\pm}$. The moduli space $\mathcal{M}( E_{f'} , \mathfrak{s} , g' , \eta')$ is a finite set of points $p'_1 , \dots , p'_{m'}$, hence we can further assume that $D$ was chosen to be disjoint from $r(p'_1) , \dots , r(p'_{m'})$. Equivalently, $r(D)$ is disjoint from $p'_1, \dots , p'_{m'}$. Let $\psi' = r \circ \psi \circ r$. We then obtain the pullback family ${\psi'}^*(E_{f'})$ with generic pair $( {\psi'}^*(g') , {\psi'}^*(\eta') )$. Moreover, we have a trivialisation of ${\psi'}^*(E_{f'})|_{r(D)}$ in which ${\psi'}^*(g')$, ${\psi'}^*(\eta')$ are sent to the constant pair $(g'_{p'} , \eta'_{p'})$, where $p' = r(p)$.

Let $D_0 \subset D$ be a smaller open disc around $p$ whose closure is contained in $D$. Attach $S^{n+1} \setminus D_0$ and $S^{n+1} \setminus r(D_0)$ to each other using a neck $[0,1] \times \partial D_0 \times [0,1]$ (note that $\partial D_0$ is diffeomorphic to $S^n$). More precisely, consider the resulting space 
\[
(S^{n+1} \setminus D_0) \cup_{\partial D_0} ([0,1] \times \partial D_0 ) \cup_{\partial D_0} (S^{n+1} \setminus r(D_0))
\]
where we identify $( 0 , y ) \in [0,1] \times \partial D_0$ with $y \in \partial( S^{n+1} \setminus D_0) = \partial D_0$ and we identify $(1 , y) \in [0,1] \times \partial D_0$ with $r(y) \in \partial( S^{n+1} \setminus r(D_0)) = \partial (r(D_0))$. Because $r$ is orientation reversing, this construction is easily seen to be the oriented connected sum of two copies of $S^{n+1}$. Of course the resulting space is just another copy of $S^{n+1}$. Using the trivialisations $\psi^*(E_f)|_{D} \cong D \times X$ and ${\psi'}^*(E_{f'})|_{r(D)} \cong r(D) \times X$, we can attach $\psi^*(E_f)|_{S^{n+1} \setminus D_0}$ to ${\psi'}^*(E_{f'})|_{S^{n+1} \setminus r(D_0)}$ by taking a constant family $([0,1] \times \partial D_0) \times X$ along the neck. 

Let $E$ denote the resulting family. Since $\psi^*(E_f)$ and ${\psi'}^*(E_{f'})$ are isomorphic to $E_f$ and $E_{f'}$ and since the neck $[0,1] \times \partial D_0$ connects the lower hemisphere in $\psi^*(E_f)$ with the upper hemisphere in ${\psi'}^*(E_{f'})$, it is clear that $E$ is isomorphic to the family obtained by applying the clutching construction to $f + f'$, where $+$ denotes the group operation on $\pi_n({\rm Diff}_0(X))$.

Next, since $d(X , \mathfrak{s}) = -(n+1) \le -2$, it follows that the moduli space of solutions to the Seiberg--Witten equations for a $1$-parameter family with fibres $(X , \mathfrak{s})$ has expected dimension $d(X,\mathfrak{s})+1 = -n < 0$. Therefore, for a generic path $(g_t , \eta_t)$ from $(g(p) , \eta(p))$ to $(g'(p') , \eta'(p'))$, there are no solutions to the Seiberg--Witten equations for $(X , \mathfrak{s} , g_t , \eta_t)$. Now we define a pair $(\tilde{g} , \tilde{\eta})$ for the family $E$ as follows. Restricted to $\psi^*(E_f)|_{S^{n+1} \setminus D_0}$, we take the pair to be $( \psi^*(g) , \psi^*(\eta))$. Restricted to ${\psi'}^*(E_{f'})|_{S^{n+1} \setminus r(D_0)}$, we take the pair to be $( {\psi'}^*(g') , {\psi'}^*(\eta'))$. Restricted to the constant family on the neck $[0,1] \times S^n $, we take the pair to be $( g_t , \eta_t)$, where $t \in [0,1]$ is the coordinate for the $[0,1]$ factor of the neck. Since there are no solutions to the Seiberg--Witten equations for $(g_t , \eta_t)$, it is clear that the moduli space for $( E , \mathfrak{s} , \tilde{g} , \tilde{\eta})$ is just the disjoint union 
\[
\mathcal{M}( \psi^*(E_f) , \mathfrak{s} , \psi^*(g) , \psi^*(\eta) ) \cup \mathcal{M}( {\psi'}^*(E_{f'}) , \mathfrak{s} , {\psi'}^*(g') , {\psi'}^*(\eta') )
\]
of the corresponding moduli spaces for $\psi^*(E_f)$ and ${\psi'}^*(E_{f'})$. Since there are no solutions to the Seiberg--Witten equations of the glued family along the neck, and since the metrics and perturbations $(g,\eta)$ and $(g' , \eta')$ of the original families were generic, it follows that $(\tilde{g} , \tilde{\eta})$ is also generic. It then follows that
\[
\int_{S^{n+1}} SW( E , \mathfrak{s} , \tilde{g} , \tilde{\eta} ) = sw_{\mathfrak{s}}(f) + sw_{\mathfrak{s}}(f').
\]
To complete the proof, it remains to show that $\int_{S^{n+1}} SW( E , \mathfrak{s} , \tilde{g} , \tilde{\eta} ) = sw_{\mathfrak{s}}(f+f')$. Since $E$ is isomorphic to the family obtained from applying the clutching construction to $f+f'$, we just need to show that the pair $(\tilde{g} , \tilde{\eta})$ lies in the canonical chamber $\in \mathcal{C}_0( E , \mathfrak{s})$.

Let $\mathcal{H}^+_{\psi^*(g)}(X)$, $\mathcal{H}^+_{{\psi'}^*(g')}(X)$ and $\mathcal{H}^+_{\tilde{g}}(X)$ be the bundles of harmonic self-dual $2$-forms for the families of metrics $\psi^*(g), {\psi'}^*(g')$, and $\tilde{g}$. Then from the construction of $\tilde{g}$, we have that $\mathcal{H}^+_{\tilde{g}}(X)$ is obtained by attaching $\mathcal{H}^+_{\psi^*(g)}(X)|_{S^{n+1} \setminus D_0}$ and $\mathcal{H}^+_{{\psi'}^*(g')}(X)|_{S^{n+1} \setminus r(D_0)}$ to the bundle $\mathcal{H}^+_{g_t}(X)$ over $[0,1] \times S^n$ whose fibre over $(t,x) \in [0,1] \times S^n$ is the space of $g_t$-self-dual harmonic $2$-forms. 

Let $\mathcal{H}^2_E(X)$ denote the local system whose fibres are the degree $2$ cohomology of the fibres of $E$ (the subscript $E$ is just to remind us which family $\mathcal{H}^2_E(X)$ comes from). Similarly define the local systems $\mathcal{H}^2_{\psi^*(E_f)}(X), \mathcal{H}^2_{{\psi'}^*(E_{f'})}(X)$. Then $\mathcal{H}^2_E(X)$ is obtained by attaching $\mathcal{H}^2_{\psi^*(E_f)}(X)|_{S^{n+1} \setminus D_0}$ and $\mathcal{H}^2_{{\psi'}^*(E_{f'})}(X)|_{S^{n+1} \setminus r(D_0)}$ to the constant local system over $[0,1] \times S^n$ with fibre $H^2( X_p ; \mathbb{R})$. Choose a maximal positive definite subspace $H \subseteq H^2(X_p ; \mathbb{R})$ and let $\rho_H : H^2(X_p ; \mathbb{R}) \to H$ be the projection. Taking the composition of inclusion and projection to $H$, we obtain isomorphisms
\[
\varphi_{\tilde{g}} : \mathcal{H}^+_{\tilde{g}}(X) \to B \times H, \quad \varphi_{\psi^*(g)} : \mathcal{H}^+_{\psi^*(g)}(X) \to B \times H, \quad \varphi_{{\psi'}^*(g')} : \mathcal{H}^+_{{\psi'}^*(g')}(X) \to B \times H.
\]
Similarly, we obtain an isomorphism $\varphi_{g_t} : \mathcal{H}^+_{g_t}(X) \to B \times H$. It is clear that the restriction of $\varphi_{\tilde{g}}$ to $S^{n+1} \setminus D_0$ agrees with $\varphi_{\psi^*(g)}$, the restriction of $\varphi_{\tilde{g}}$ to $S^{n+1} \setminus r(D_0)$ agrees with $\varphi_{{\psi'}^*(g')}$ and the restriction of $\varphi_{\tilde{g}}$ to $[0,1] \times S^n$ agrees with $\varphi_{g_t}$.

Let $w : S^{n+1} \to \mathcal{H}^+_{\tilde{g}}(X)$ denote the wall for the family $E$ and set
\[
R = \sup_{b \in S^{n+1}} || \varphi_{\tilde{g}}( w_b ) ||_H.
\]
Recall that we have assumed $(g,\eta) \in \mathcal{C}_0( E_f , \mathfrak{s})$. Fix an element $v \in H$ such that $|| v ||_H > R$. Choose an $\epsilon > 0$ such that each $u$ in the open ball $B(v,\epsilon) = \{ u \in H \; | \; || u-v||_H < \epsilon \}$ has $||u||_H > R$ ($\epsilon = (||v||_H - R)/2$ would suffice). We will assume that $\eta$ is chosen with $\varphi_{g_b}[\eta]_b \in B(v,\epsilon)$ for all $b \in S^{n+1}$. Then we also have $\varphi_{\psi^*(g)}[\psi^*(\eta)]_b \in B(v,\epsilon)$ for all $b \in S^{n+1}$. Similarly, we can assume that $\eta'$ was chosen so that $\varphi_{g'_b}[\eta']_b \in B(v,\epsilon)$ for all $b \in S^{n+1}$ and hence $\varphi_{{\psi'}^*(g')}[{\psi'}^*(\eta')]_b \in B(v,\epsilon)$ for all $b \in S^{n+1}$ as well. Lastly, we can assume that the generic path $(g_t, \eta_t)$ joining $(g_p,\eta_p)$ to $(g'_p , \eta'_p)$ satisfies $\varphi_{g_t}( [\eta_t] ) \in B(v,\epsilon)$ for all $t \in [0,1]$. It follows that $\varphi_{\tilde{g}} [\tilde{\eta}]_b \in B(v,\epsilon)$ for all $b \in S^{n+1}$. Therefore we can find a homotopy from $\varphi_{\tilde{g}}[\tilde{\eta}]$ to the constant section $v$ and hence $(\tilde{g} , \tilde{\eta})$ lies in $\mathcal{C}_0(E , \mathfrak{s})$.
\end{proof}

Let ${\rm Diff}(X)$ act on itself by conjugation. Since the identity element is fixed, this gives an action of ${\rm Diff}(X)$ on $\pi_n( {\rm Diff}_0(X) )$. We write this action as $(f , h ) \mapsto f h f^{-1}$.

\begin{proposition}\label{prop:diff}
Let $X$ be a compact, oriented, smooth $4$-manifold with $b^+(X) > 1$ and let $f : X \to X$ be an orientation preserving diffeomorphism. Then for each spin$^c$-structure with $d(X,\mathfrak{s})=-(n+1) \le -2$ we have
\[
sw_{\mathfrak{s}}( f h f^{-1}) = sw_{f^*(\mathfrak{s})}(h).
\]
\end{proposition}
\begin{proof}
Let $D_+, D_-$ be two copies of the unit disc in $\mathbb{R}^{n+1}$. Attaching $D_+$ and $D_-$ along their boundary gives $S^{n+1}$. Let $h \in \pi_n( {\rm Diff}_0(X))$. The family $E_h$ is obtained by attaching $D_+ \times X$ to $D_- \times X$ using the attaching map $\partial D_- \times X \to \partial D_+ \times X$, $(b,x) \mapsto a_h(b,x) = (b , (h(b))(x))$. Similarly $E_{fhf^{-1}}$ is constructed using $(b,x) \mapsto a_{fhf^{-1}}(b,x) =  (b , (f h(b)f^{-1})(x))$. Consider the maps $\tilde{f}_{\pm} : D_{\pm} \times X \to D_{\pm} \times X$ given by $\tilde{f}_{\pm}(b,x) = (b , f(x))$. One finds that
\[
a_{f h f^{-1}} \circ \tilde{f}_- = \tilde{f}_+ \circ a_h.
\]
This says that the maps $\tilde{f}_\pm$ glue together to define a map $\tilde{f} : E_h \to E_{f h f^{-1}}$. The map $\tilde{f}$ is an isomorphism of smooth families over $S^{n+1}$. Let $\tilde{s}$ be the continuous extension of $\tilde{s}$ to a family of spin$^c$-structures on the fibres of $E_{f h f^{-1}}$. Then clearly $\tilde{f}^*( \tilde{\mathfrak{s}})$ is a continuous extension of $f^*(\mathfrak{s})$ to a family of spin$^c$-structures on the fibres of $E_h$. Let $(g,\eta)$ be a generic pair for $(E_{fhf^{-1}} , \mathfrak{s})$ lying in the canonical chamber. Then $( \tilde{f}^*(g) , \tilde{f}^*(\eta))$ is a generic pair for $(E_h , f^*(\mathfrak{s})$ lying in the canonical chamber. Clearly $\tilde{f}$ induces an isomorphic between the corresponding moduli spaces for $(E_{fhf^{-1}} , \mathfrak{s} , g , \eta)$ and $(E_h , f^*(\mathfrak{s}) , \tilde{f}^*(g) , \tilde{f}^*(\eta) )$. Hence $sw_{\mathfrak{s}}( f h f^{-1}) = sw_{f^*(\mathfrak{s})}(h)$.
\end{proof}

Recall that there is an involution $\mathfrak{s} \mapsto \overline{\mathfrak{s}}$ on the set of spin$^c$-structure which we refer to as charge conjugation \cite[Page 51]{nic}. Recall that $c_1(\overline{\mathfrak{s}}) = -c_1( \mathfrak{s})$.

\begin{proposition}\label{prop:conj}
Let $X$ be a compact, oriented, smooth $4$-manifold with $b^+(X) > 1$ and let $\mathfrak{s}$ be a spin$^c$-structure with $d(X,\mathfrak{s})=-(n+1) \le -2$. Then
\[
sw_{\overline{\mathfrak{s}}} = (-1)^{\frac{b_+(X)-b_1(X)-n}{2}} sw_{\mathfrak{s}}.
\]
\end{proposition}
\begin{proof}
Recall that charge conjugation gives rise to a bijection from the Seiberg--Witten equations for $(X , \mathfrak{s} , g , \eta)$ to the Seiberg--Witten equations for $(X , \overline{\mathfrak{s}} , g , -\eta)$. Fix a homology orientation for $X$, giving orientations on $\mathcal{M}(X , \mathfrak{s} , g , \eta)$ and $\mathcal{M}( X , \overline{\mathfrak{s}} , g , -\eta)$. The charge conjugation map is orientation preserving or reversing according to the sign of $(-1)^{d_\mathfrak{s} + 1 - b_1(X) + b^+(X)}$ where 
\[
d_\mathfrak{s} = \frac{ c_1(\mathfrak{s})^2 - \sigma(X) }{8},
\]
see \cite[Proposition 2.2.26]{nic}. Similarly, charge conjugation gives rise to a bijection of families moduli spaces. By a straighforward extension of \cite[Proposition 2.2.26]{nic} to the families setting, we see that the charge conjugation isomorphism changes the orientation of the families moduli space by the same factor $(-1)^{d_\mathfrak{s}+1-b_1(X) + b^+(X) }$. Moreover, it is clear that if $(g , \eta)$ is in the canonical chamber for $(E , \mathfrak{s})$, then $(g , -\eta)$ is in the canonical chamber for $(E , \overline{\mathfrak{s}})$. Hence
\[
sw_{\overline{\mathfrak{s}}} = (-1)^{u} sw_{\mathfrak{s}}
\]
where
\[
u = d_\mathfrak{s} + 1 - b_1(X) + b^+(X).
\]
But since
\[
-n-1 = d(X , \mathfrak{s}) = 2 d_\mathfrak{s} + 1 - b_1(X) + b^+(X)
\]
we see that
\[
d_\mathfrak{s} = \frac{-n-2+b_1(X)-b^+(X)}{2}
\]
and hence
\[
u = \frac{-n-2+b_1(X)-b^+(X)}{2} + 1-b_1(X)+b^+(X) = \frac{b^+(X) - b_1(X) -n }{2}.
\]
\end{proof}

\begin{theorem}\label{thm:finite}
Let $X$ be a compact, oriented, smooth $4$-manifold such that $b^+(X) > 1$ and $b_1(X) = 0$. For a given $f \in \pi_n({\rm Diff}_0(X))$, we have that $sw_{\mathfrak{s}}(f)$ is non-zero for only finitely many spin$^c$-structures with $d(X,\mathfrak{s}) = -(n+1)$.
\end{theorem}
\begin{proof}

Let $g$ be a metric on $X$ and consider the Seiberg--Witten equations on $X$ with respect to the metric $g$ and zero perturbation. Recall that the a priori estimates for solutions $(A , \psi)$ of the Seiberg--Witten equations (after gauge fixing) imply bounds on the norms of $A, \psi$ in a suitable Sobolev space \cite[\textsection 2.2]{nic}. A bound $M(g)$ can be chosen which depends continuously on $g$ and the topology of $X$, but does not depend on the spin$^c$-structure $\mathfrak{s}$. Hence for a smooth family $E \to B$ over a compact base $B$, we obtain compactness of the families Seiberg--Witten moduli space, taken over {\em all} spin$^c$-structures, with zero perturbation and a fixed family $g = \{g_b\}$ of metrics. It follows that the families moduli space $\mathcal{M}( E , \mathfrak{s} , g , 0)$ is non-empty for only finitely many spin$^c$-structures, say, $\mathfrak{s}_1, \dots , \mathfrak{s}_m$. For any other spin$^c$-structure, $\mathfrak{s}$, the moduli space $\mathcal{M}(E , \mathfrak{s} , g , 0 )$ is empty. Hence $\eta = 0$ is a generic perturbation for $(E , \mathfrak{s})$ and $(g , 0)$ defines a chamber for $(E , \mathfrak{s})$.

Now let $f \in \pi_n({\rm Diff}_0(X))$ and take $E \to B$ to be the family $E_f \to S^{n+1}$ associated to $f$. If $b^+(X) > n+1$, then there is only one chamber and hence we deduce that $sw_{\mathfrak{s}}(f) = 0$ for all but finitely many spin$^c$-structures.

If $b^+(X) \le n+2$, then we have to consider chambers. Fix a family of metrics $g$. Then we have shown that for all but finitely many spin$^c$ structures $\mathfrak{s}$, $(g , 0)$ is a generic perturbation and $SW(E_f , \mathfrak{s} , g , 0 ) = 0$. However, $(g,0)$ might not lie in the canonical chamber. Hence we need to consider contributions to the Seiberg--Witten invariant from wall crossing.

For the rest of the proof, the family $E_f$ and metric $g$ will be fixed. To simplify notation we will write $SW( \mathfrak{s} , \eta)$ instead of $SW(E_f , \mathfrak{s} , g , \eta)$, whenever $\eta$ is a generic perturbation for $(E , \mathfrak{s} , g)$. Fix a maximal positive definite subspace $H$ of $H^2(X ; \mathbb{R})$. Let $\varphi : \mathcal{H}^+_g(X) \to H$ be the map which is the inclusion of $\mathcal{H}^+_g(X)$ into $H^2(X ; \mathbb{R})$, followed by projection to $H$. For each spin$^c$-structure $\mathfrak{s}$, let $w(\mathfrak{s})$ be the section of $H$ which sends $b \in B$ to $\varphi( 2\pi c_1(\mathfrak{s})^{+_{g_b}} )$. Given a perturbation $\eta$, let $w(\eta)$ be the section of $H$ given by $b \mapsto \varphi( [\eta]_b)$. If $w(\eta)$ and $w(\mathfrak{s})$ are disjoint, then $(g,\eta)$ defines a chamber for $(E , \mathfrak{s})$ and hence the Seiberg--Witten invariant $SW(\mathfrak{s},\eta)$ is defined.

Let $S(H)$ denote the unit sphere in $H$, which has dimension $b^+(X)-1$. If $w(\eta)$ and $w(\mathfrak{s})$ are disjoint, then $\phi = (w(\eta) - w(\mathfrak{s}))/|| w(\eta) - w(\mathfrak{s}) ||_H$ defines a section of $S(H)$. The wall crossing formula for the families Seiberg--Witten invariants \cite[Corollary 5.5]{bako2} (see also \cite{liliu}) adapted to the present setting states that
\begin{equation}\label{equ:obs1}
SW(\mathfrak{s} , \eta_1) - SW(\mathfrak{s} , \eta_2) = Obs ( \phi , \psi) s_{1-d}(D)
\end{equation}
where $\phi, \psi : B \to S(H)$ are the sections of $S(H)$ given by
\[
\phi = \frac{w(\eta_1) - w(\mathfrak{s})}{||w(\eta_1) - w(\mathfrak{s})||_H}, \quad \psi = \frac{w(\eta_2) - w(\mathfrak{s})}{||w(\eta_2) - w(\mathfrak{s})||_H},
\]
$s_k(D)$ is the $k$-th Segre class of $D$, the families index of the family of spin$^c$ Dirac operators determined by $(E , \mathfrak{s} , g)$, $d = (c_1(\mathfrak{s}^2 - \sigma(X))/8$ and $Obs(\phi , \psi) \in H^{b^+(X)-1}( B ; \mathbb{Z})$ is the primary difference class of $\phi,\psi$ (\cite[\textsection 36]{ste}), the obstruction to constructing a homotopy of two maps $\phi, \psi : B \to S(H)$ over the $b^+(X)-1$ skeleton of $B$. In our case $B = S^{n+1}$ and so $H^{b^+(X)-1}(B ; \mathbb{Z}) = H^{b^+(X)-1}( S^{n+1} ; \mathbb{Z})$ is zero unless $b^+(X) = n+2$.

If $b^+(X) \neq n+2$ then the primary obstruction vanishes implying that the value of $SW(\mathfrak{s} , \eta)$ does not depend on the choice of chamber. But we have already seen that for all but finitely many spin$^c$ structures $\mathfrak{s}$, $SW(\mathfrak{s} , 0) = 0$. Hence $sw_{\mathfrak{s}} = 0$ for all but finitely many $\mathfrak{s}$.

It remains to consider the case $b^+(X) = n+2$. In this case we have $-n-1 = d(X , \mathfrak{s}) = 2d -n-3$ and hence $d=1$. But $s_0(D) = 1$ and so Equation (\ref{equ:obs1}) reduces to $SW(\mathfrak{s} , \eta_1) - SW(\mathfrak{s} , \eta_2) = Obs ( \phi , \psi)$. Furthermore, the primary obstruction is valued in $H^{n+1}( S^{n+1} ; \mathbb{Z}) \cong \mathbb{Z}$. Let $\nu \in H^{n+1}( S^{n+1} ; \mathbb{Z})$ be the generator corresponding to our chosen orientation on $S^{n+1}$. Then from \cite[Proposition 5.7]{bako2}, we have
\[
Obs(\phi , \psi ) = (-1)^{b^+(X)-1}( \phi^*(\nu) - \psi^*(\nu) ).
\]
Integrating over $S^{n+1}$, the wall crossing formula reduces to
\[
\int_{S^{n+1}} SW(\mathfrak{s} , \eta_1) - \int_{S^{n+1}} SW(\mathfrak{s} , \eta_2) = (-1)^{b^+(X)-1}( {\rm deg}(-\phi_{\mathfrak{s}, \eta_1}) - {\rm deg}(-\phi_{\mathfrak{s},\eta_2}) ),
\]
where we define
\[
\phi_{\mathfrak{s} , \eta} = \frac{ w(\mathfrak{s}) - w(\eta) }{ ||w(\mathfrak{s}) -w(\eta)||_H }
\]
for any perturbation $\eta$ such that $w(\eta)$ and $w(\mathfrak{s})$ are disjoint. 

Noting that ${\rm deg}( -\phi ) = (-1)^{b^+(X)} {\rm deg}(\phi)$, the wall crossing formula can be re-written as
\[
\int_{S^{n+1}} SW(\mathfrak{s} , \eta_1) - \int_{S^{n+1}} SW(\mathfrak{s} , \eta_2) = -{\rm deg}(\phi_{\mathfrak{s}, \eta_1}) + {\rm deg}(\phi_{\mathfrak{s},\eta_2}).
\]

Now let us take $\eta_1 = \eta$ to be arbitrary and choose $\eta_2$ such that $w(\eta_2) = v$ is a constant such that $|| v ||_H > \sup_{B} || w(\mathfrak{s})||_H$. Then $(g , \eta_2)$ lies in the canonical chamber. Now since $|| w(\eta_2) ||_H > || w(\mathfrak{s}) ||_H$ for all $b \in B$, we obtain a homotopy
\[
t \mapsto \frac{ (1-t) w(\mathfrak{s}) - w(\eta_2) }{|| (1-t)w(\mathfrak{s}) - w(\eta_2) ||_H}, \quad t \in [0,1]
\]
from $\phi_{\mathfrak{s} , \eta_2}$ to the constant $-v/||v||_H$. It follows that ${\rm deg}( \phi_{\mathfrak{s} , \eta_2} ) = 0$ and therefore
\[
\int_{S^{n+1}} SW(\mathfrak{s} , \eta) - \int_{S^{n+1}} SW(\mathfrak{s} , \eta_2) = -{\rm deg}(\phi_{\mathfrak{s}, \eta}).
\]
But $(g , \eta_2)$ lies in the canonical chamber, so $\int_{S^{n+1}} SW(\mathfrak{s} , \eta_2) = sw_{\mathfrak{s}}(f)$. Hence the above formula reduces to
\begin{equation}\label{equ:wcf}
\int_{S^{n+1}} SW(\mathfrak{s} , \eta) = sw_{\mathfrak{s}}(f) - {\rm deg}( \phi_{\mathfrak{s} , \eta} ).
\end{equation}
Now we set $\eta=0$. Then for all but finitely many $\mathfrak{s}$, we have that $w(\mathfrak{s})$ is non-vanishing and that $SW(\mathfrak{s} , 0) = 0$. Hence for all but finitely many $\mathfrak{s}$, we find
\[
sw_{\mathfrak{s}}(f) = {\rm deg}( w(\mathfrak{s})/||w(\mathfrak{s})||_H ).
\]

To finish the proof, it remains to show that when $b^+(X) = n+2$, there are only finitely many $\mathfrak{s}$ such that $d(X,\mathfrak{s}) = -(n+1)$, $w(\mathfrak{s})$ is non-vanishing and $w(\mathfrak{s})/||w(\mathfrak{s})||_H : S^{n+1} \to S(H)$ has non-zero degree. For convenience, let us say that a spin$^c$-structure $\mathfrak{s}$ is {\em valid} if $d(X,\mathfrak{s}) = -(n+1)$ and $w(\mathfrak{s})$ is non-vanishing and let us write ${\rm deg}( w(\mathfrak{s}) )$ for the degree of $w(\mathfrak{s})/||w(\mathfrak{s})||_H$. Then we need to show that ${\rm deg}( w(\mathfrak{s}) ) = 0$ for all but finitely many valid $\mathfrak{s}$.

We first show that there is a constant $\kappa$ such that ${\rm deg}( w(\mathfrak{s}) ) = \kappa$ for all but finitely many valid $\mathfrak{s}$. We will then argue that $\kappa = 0$.

Note that if $b^+(X) = n+2$ and $d(X,\mathfrak{s}) = -(n+1)$, then
\[
\frac{ c_1(\mathfrak{s})^2 - \sigma(X) }{4} - n - 3 = -n-1
\]
(where we used $b_1(X)=0$) and hence if $\mathfrak{s}$ is valid, then
\[
c_1(\mathfrak{s})^2 = \sigma(X) + 8.
\]
We set $N = \sigma(X)+8$. If $N \ge 0$, then for every non-zero $c \in H^2(X ; \mathbb{R})$ such that $c^2 = N$, the orthogonal projection $c^{+_g}$ of $c$ to $H^+_g(X)$ is non-zero. This is because
\[
(c^{+_g})^2 \ge (c^{+_g})^2 - |(c^{-_g})^2| = c^2 = N \ge 0
\]
and equality $(c^{+_g})^2 = 0$ can only occur if $c=0$. So if $N \ge 0$, then every non-zero $c \in H^2( X ; \mathbb{R})$ defines a non-zero map $w(c) : S^{n+1} \to H$ by taking $w(c) = \varphi( 2 \pi c^{+_g})$. The set $\{ c \in H^2( X \; \mathbb{R} ) \; | \; c \neq 0, \; c^2 = N \}$ is clearly connected if $N>0$, since $b^+(X)  > 1$. Also if $N=0$, then $\sigma(X) = -8$ and so $b^+(X), b^-(X) > 1$. It follows that $\{ c \in H^2( X \; \mathbb{R}) \; | \; c \neq 0, \; c^2 = 0\}$ is connected. Therefore the degree of $w(c)/||w(c)||_H$ is a constant $\kappa$. Now there are only finitely many spin$^c$ structures $\mathfrak{s}$ for which $c_1(\mathfrak{s}) = 0$, hence $c_1(\mathfrak{s}) \in \{c \in H^2( X ; \mathbb{R} ) \; | \; c \neq N \}$ for all but finitely many valid $\mathfrak{s}$. So ${\rm deg}( w(\mathfrak{s}) ) = \kappa$ for all but finitely many valid $\mathfrak{s}$.

Now we suppose $N < 0$. So $\sigma(X) < -8$ and in particular $b^-(X) > b^+(X) > 1$. Let us define
\[
C_N = \{ c \in H^2( X ; \mathbb{R}) \; | \; c^2 = N \}.
\]
Then $C_N$ is homotopy equivalent to a sphere of dimension $b^-(X)-1$. For each $b \in B$, consider
\[
S_b = \{ c \in C_N \; | \; c^{+_{g_b}} = 0 \}.
\]
The condition $c^{+_{g_b}} = 0$ means that $c$ lies in the negative definite subspace $H^-_{g_b}(X)$. Therefore $S_b$ is a sphere of dimension $b^-(X)-1$. In particular $S_b$ is compact. Similarly, let
\[
S = \bigcup_{b \in B} S_b = \{ c \in C_N \; | \; c^{+_{g_b}} = 0 \text{ for some } b \in B \}.
\]
Then $S$ is a compact subset of $C_N$ (by compactness of $B$). Choose an isometry $H^2(X ; \mathbb{R}) \cong \mathbb{R}^{r,s}$ where $r = b^+(X) > 1$ and $s = b^-(X)>1$. We can further identify $\mathbb{R}^{r,s}$ with $\mathbb{R}^r \oplus \mathbb{R}^s$ with bilinear form $\langle (x_1 , y_1) , (x_2 , y_2) \rangle_{r,s} = \langle x_1 , x_2 \rangle - \langle y_1 , y_2 \rangle$, where $\langle x_1 , x_2 \rangle$ and $\langle y_1 , y_2 \rangle$ are the standard inner products on $\mathbb{R}^r$ and $\mathbb{R}^s$. Then if $c = (x,y) \in H^2(X ; \mathbb{R}) \cong \mathbb{R}^{r+s}$, we have $c^2 = x^2 - y^2$, where $x^2 = \langle x , x \rangle$ and $y^2 = \langle y , y \rangle$. It follows that $C_N \cong \{ (x,y) \in \mathbb{R}^{r+s} \; | \; x^2 - y^2 = N \}$. Define a Euclidean norm $|| \; ||_{E}$ on $H^2( X ; \mathbb{R})$ by setting $|| c ||_{E}^2 = x^2 + y^2$. By compactness of $S$, we have that $S$ is contained in some ball $B_R = \{ x \in H^2(X ; \mathbb{R}) \; | \; ||x ||_E \le R \}$ of sufficiently large radius $R>0$. Then $C_N \setminus (C_N \cap B_R)$ may be identified with
\[
\{ (x,y) \in \mathbb{R}^{r+s} \; | \; x^2 - y^2 = N, \; x^2 + y^2 > R^2 \}.
\]
Equivalently, this is the set of pairs $(x,y)$ such that $x^2 > (R^2+N)/2$ and $y^2 = x^2 - N$. Recall that $N < 0$. Hence if $R^2 > -N$, we see that this space is homotopy equivalent to $S^{b^+(X)-1} \times S^{b^-(X)-1}$, which is connected as $b^+(X), b^-(X) > 1$. For any $c \in C_N \setminus (C_N \cap B_R)$, define $w(c) : S^{n+1} \to H$ as $w(c) = \varphi( 2 \pi c^{+_g} )$. Then since $c \notin B_R$, we have that $w(c)$ is non-vanishing and hence the degree of $w(c)/||w(c)||_H$ is defined. Since $C_N \setminus (C_N \cap B_R)$ is connected, the degree of $w(c)/||w(c)||_H$ is equal to a constant, $\kappa$, for every $c \in C_N \setminus (C_N \cap B_R)$.

Let $\mathfrak{s}$ be a valid spin$^c$-structure. Then $c_1(\mathfrak{s}) \in C_N$. If $c_1(\mathfrak{s}) \notin B_R$, then it follows that ${\rm deg}( w(\mathfrak{s}) ) = \kappa$. Next, we note that if $c_1(\mathfrak{s}) \in B_R$, then $c_1(\mathfrak{s}) \in B_R \cap H^2(X ; \mathbb{Z})$. But $B_R \cap H^2(X ; \mathbb{Z})$ is finite because $B_R$ is compact and $H^2(X ; \mathbb{Z})$ is discrete. It follows that for all but finitely many valid $\mathfrak{s}$, we have ${\rm deg}(w(\mathfrak{s})) = \kappa$.

Now we argue that $\kappa = 0$. We showed that for all $c \in C_N$ outside some ball $B_R$, the degree of $w(c)$ is $\kappa$. Let $\psi : H^2(X ; \mathbb{R}) \to H^2(X ; \mathbb{R})$ be an isometry of the intersection form on $H^2(X ; \mathbb{R})$ that sends $H$ to itself and reverses orientation on $H$. Then $B_R \cup \psi(B_R)$ is compact so there exists a $c \in C_N$ such that $c \notin B_R \cup \psi(B_R)$. Hence $c,\psi(c) \in C_N \setminus (C_N \cap B_R)$. Therefore 
\[
{\rm deg}( w(c) ) = {\rm deg}( w( \psi(c) ) ) = \kappa.
\]
On the other hand, since $\psi$ reverses orientation on $H$, we have 
\[
{\rm deg}( w(\psi(c))) = -{\rm deg}( w(c) ) = -\kappa.
\]
This gives $\kappa = -\kappa$, hence $\kappa = 0$. Thus ${\rm deg}( w(\mathfrak{s})) = 0$ for all but finitely many valid $\mathfrak{s}$ and the proof is complete.
\end{proof}

\section{The Einstein family}\label{sec:ein}

Let $X$ be the underlying oriented smooth $4$-manifold of a complex $K3$ surface. We will use the moduli space of Einstein metrics on $X$ to construct non-trivial families. The construction of this moduli space follows \cite{gia}, \cite{gkt} and \cite{bako3}.

Let $Ein$ denote the space of all Einstein metrics on $X$ with unit volume given the $\mathcal{C}^\infty$-topology. Every Einstein metric on $X$ is a hyper-k\"ahler metric \cite{hit} and in particular, Ricci flat. Then ${\rm Diff}(X)$ acts on $Ein$ by pullback. That is, for each $\varphi \in {\rm Diff}(X)$ we define
\[
\varphi_* : Ein \to Ein, \quad \varphi_*(g) = (\varphi^{-1})^*(g).
\]
Let ${\rm TDiff}(X)$ denote the subgroup of ${\rm Diff}(X)$ consisting of those diffeomorphisms which act trivially on $H^2(X ; \mathbb{Z})$. Let $Aut( H^2( X ; \mathbb{Z}))$ be the group of automorphisms of the intersection form. Then we have a short exact sequence
\[
1 \to {\rm TDiff}(X) \to {\rm Diff}(X) \to \Gamma \to 1,
\]
where $\Gamma \subset Aut( H^2( X ; \mathbb{Z}))$ is the subgroup of automorphisms that are induced by diffeomorphisms of $X$. Note that since $\Gamma$ is discrete, the identity components of ${\rm TDiff}(X)$ and ${\rm Diff}(X)$ are the same
\begin{equation}\label{equ:tdiff0}
{\rm TDiff}_0(X) = {\rm Diff}_0(X).
\end{equation}

As a consequence of the global Torelli theorem for $K3$ surfaces, one finds that ${\rm TDiff}(X)$ acts freely and properly on $Ein$ \cite[\textsection 4]{gia}. Let
\[
T_{Ein} = Ein / {\rm TDiff}(X)
\]
be the quotient. Over $Ein$ we have the constant family $X \times Ein \to Ein$. We can equip the vertical tangent space of $X \times Ein$ with the tautological metric, namely the metric on the fibre $X \times \{ g\}$ is $g$ itself. The action of ${\rm Diff}(X)$ on $Ein$ lifts to $X \times Ein$ by setting
\[
\varphi_* ( x , g ) = ( \varphi(x) , \varphi_*(g) ).
\]
It is easily checked that this action preserves the fibrewise metric. Let $E_{Ein} = ( X \times Ein)/{\rm TDiff}(X)$ be the quotient of $X \times Ein$ by the action of ${\rm TDiff}(X)$. 
\begin{lemma}
$E_{Ein}$ is a locally trivial smooth family over $T_{Ein}$ with fibres diffeomorphic to $X$. 
\end{lemma}
\begin{proof}
Let $Met$ denote the space of all metrics on $X$ with the $\mathcal{C}^\infty$ topology. The Ebin slice theorem \cite[Theorem 7.4]{eb}, \cite[Slice Theorem]{ck} implies that $Met/{\rm TDiff}(X)$ has the structure of an infinite dimensional orbifold. Namely if $g \in Met$ and $S_g \subset Met$ is a local slice around $g$ as given by the Ebin slice theorem, then the isometry group $I(g)$ of $g$ preserves $S_g$ and the quotient $S_g/(I(g) \cap {\rm TDiff(X)})$ may be identified with a neighbourhood of $[g] \in Met/{\rm TDiff}(X)$. In particular, the isotropy group of $[g]$ is $I(g) \cap {\rm TDiff}(X)$. The quotient $\mathcal{U} = (X \times Met)/{\rm TDiff}(X)$ is a smooth orbifold bundle over $Met/{\rm TDiff}(X)$. Away from the singular points of $Met/{\rm TDiff}(X)$, it is a smooth fibre bundle with fibre $X$. The inclusion $Ein \to Met$ descends to an inclusion $\iota : T_{Ein} \to Met/{\rm TDiff}(X)$ whose image is disjoint from the singuarities. This is because if $g$ is an Einstein metric on $X$, then $I(g) \cap {\rm TDiff}(X) = 1$ \cite[Lemma 4.4]{gia}. The deformation theory of Einstein metrics around K\"ahler--Einstein metrics \cite[Theorem 10.5, Corollary 3.5]{koi} implies that $T_{Ein}$ is a finite-dimensional, smoothy embedded submanifold of $Met/{\rm TDiff}(X)$. So the pullback $E_{Ein} = \iota^* \mathcal{U}$, is a smooth, locally trivial fibre bundle over $T_{Ein}$ with fibres diffeomorphic to $X$.
\end{proof}

The tautological metric on $X \times Ein$ descends to a metric $g_{Ein}$ on the vertical tangent bundle such that the restriction of $g_{Ein}$ to the fibre of $E_{Ein}$ over $[g] \in T_{Ein}$ is a representative of the isomorphism class of Einstein metrics $[g]$. Thus $E_{Ein}$ is a family of Einstein metrics on $X$.

Let $Gr_3( \mathbb{R}^{3,19})$ denote the Grassmannian of positive definite $3$-planes in $\mathbb{R}^{3,19}$. There is a period map
\[
P : T_{Ein} \to Gr_3( \mathbb{R}^{3,19})
\]
defined as follows. Fix an isometry $H^2(X ; \mathbb{R}) \cong \mathbb{R}^{3,19}$. Then $P$ sends an Einstein metric $g$ to the $3$-plane $H^+_g(X)$. Let
\[
\Delta = \{ \delta \in H^2(X ; \mathbb{Z}) \; | \; \delta^2 = -2 \}
\]
and set
\[
W = \{ H \in Gr_3(\mathbb{R}^{3,19}) \; | \; H^\perp \cap \Delta = \emptyset \}.
\]
The Grassmannian $Gr_3(\mathbb{R}^{3,19})$ is a contractible manifold (it is a symmetric space of non-compact type) and for each $\delta \in \Delta$, the subset
\[
A_\delta = \{ H \in Gr_3(\mathbb{R}^{3,19}) \; | \; \delta \in H^\perp \}
\]
is a codimension $3$ embedded submanifold and $W = Gr_3(\mathbb{R}^{3,19}) \setminus \bigcup_{\delta \in \Delta} A_\delta$. It follows from the global Torelli theorem for $K3$ that the period map $P$ is a homeomorphism of $T_{Ein}$ to the set $W$ \cite[Chapter 12, K]{bes}.

Let $Gr_3^+(\mathbb{R}^{3,19})$ denote the set of pairs $(H , \mathfrak{o})$ where $H \in Gr_3(\mathbb{R}^{3,19})$ is a positive definite $3$-plane and $\mathfrak{o}$ is an orientation on $H$. The forgetful map $Gr_3^+(\mathbb{R}^{3,19}) \to Gr(\mathbb{R}^{3,19})$ is a double covering space. However $Gr_3(\mathbb{R}^{3,19})$ is contractible and so $Gr_3^+(\mathbb{R}^{3,19})$ is the trivial double covering consisting of two copies of $Gr_3(\mathbb{R}^{3,19})$. Let $g$ be an Einstein metric on $X$. Then $g$ is a hyper-k\"ahler metric with holonomy group equal to $Sp(1)$. Let $I,J,K$ be a hyper-k\"ahler triple of complex structures for $g$. Let $\omega_I, \omega_J, \omega_K$ be the correspondinf K\"ahler forms. Then $\{ \omega_I , \omega_J , \omega_K\}$ defined an oriented basis for $H^+_g(X)$. Since $g$ has holonomy $Sp(1)$, the triple $\omega_I , \omega_J , \omega_K$ is determined up to an $SO(3)$ transformation. Hence we obtain a canonical orientation on $H^+_g(X)$. Furthermore, if two Einstein metrics $g,g'$ define the same point in the moduli space $T_{Ein}$ so that $H^+_g(X) = H^+_{g'}(X)$, then the induced orientations are the same \cite[Lemma 2.1]{bako3}. This means that the period map $P : T_{Ein} \to Gr_3(\mathbb{R}^{3,19})$ admits a canonical lift $\widetilde{P} : T_{Ein} \to Gr_3^+(\mathbb{R}^{3,19})$. Furthermore since $T_{Ein}$ is path-connected \cite[Page 4]{bako3}, it follows that the image of $\widetilde{P}$ is contained in a single component of $Gr_3^+(\mathbb{R}^{3,19})$. Let us denote this component by $Gr_3'(\mathbb{R}^{3,19})$. The forgetful map $Gr_3^+(\mathbb{R}^{3,19}) \to Gr(\mathbb{R}^{3,19})$ restricted to $Gr_3'(\mathbb{R}^{3,19})$ gives an homeomorphism $Gr_3'(\mathbb{R}^{3,19}) \cong Gr_3(\mathbb{R}^{3,19})$ and in this way, every $H \in Gr_3(\mathbb{R}^{3,19})$ inherits an orientation.

Choose an element $v \in H^2(X ; \mathbb{R})$ such that $\langle v , \delta \rangle \neq 0$ for all $\delta \in \Delta$ and define $\Delta^{\pm} = \{ \delta \in \Delta \; | \; \pm \langle v , \delta \rangle > 0 \}$. Then
\[
\Delta = \Delta^+ \cup \Delta^-
\]
and $\delta \in \Delta^+$ if and only if $-\delta \in \Delta^-$.

Let $\delta \in \Delta^+$ and choose a point $p \in A_\delta$ such that $p$ does not lie on any $A_{\delta'}$ for $\delta' \in \Delta^+$ other than $\delta$ (each $A_{\delta'}$ intersects $A_\delta$ in a closed embedded submanifold of positive codimension. The set $\Delta$ is countable, so $\{ A_{\delta'} \}_{\delta' \in \Delta \setminus\{\pm \delta\}}$ does not cover $A_\delta$). Let $H \subset H^2(X ; \mathbb{R})$ be the positive definite $3$-plane corresponding to $p$. As explained above $H$ can be given a canonical orientation. Choose an oriented basis $\theta_1, \theta_2, \theta_3$ for $H$ satisfying $\langle \theta_i , \theta_j \rangle = \delta_{ij}$. Since $p \in A_\delta$, we have $\langle \theta_j , \delta \rangle = 0$ for $j = 1,2,3$. Moreover, since $p$ does not lie on $A_{\delta'}$ for any $\delta' \in \Delta^+$ not equal to $\delta$, we have $\langle \theta_j , \delta' \rangle \neq 0$ for some $j$.

Let $B = S^2 = \{ (x_1,x_2,x_3) \in \mathbb{R}^3 \; | \; x_1^2 + x_2^2 + x_3^2 = 1 \} \subset \mathbb{R}^3$ be the unit $2$-sphere and choose an $\epsilon \in (0,1)$. Consider the map 
\[
f_\delta : S^2 \to Gr_3( \mathbb{R}^{3,19})
\]
defined by
\[
f_\delta(x_1,x_2,x_3) = {\rm span}( \omega_1 , \omega_2 , \omega_3 )
\]
where for $i=1,2,3$, we set
\[
\omega_i = \theta_i - \epsilon x_i \delta/2.
\]
We choose $\epsilon$ sufficiently small so that $f_\delta(x_1,x_2,x_3)$ is a positive definite subspace of $H^2(X ; \mathbb{R})$.

\begin{lemma}
If $\epsilon$ is sufficiently small then $f_\delta(x_1,x_2,x_3)$ does not lie on $A_{\delta'}$ for any $\delta' \in \Delta^+ \setminus \{\delta\}$.
\end{lemma}
\begin{proof}
Consider the decomposition $\mathbb{R}^{3,19} \cong H \oplus K$, where $K = H^\perp$. Then any $x \in \mathbb{R}^{3,19}$ uniquely decomposes as $x = x_H + x_K$, where $x_H \in H$, $x_K \in K$. Let $\alpha \in \Delta^+ \setminus \{\delta\}$. We first show that if $\epsilon^2 < 4 (\alpha_H)^2/\sqrt{2(\alpha_H^2+2)}$, then $f_\delta(x_1,x_2,x_3) \notin A_\alpha$ for any $(x_1,x_2,x_3) \in S^3$. To see this, note that $f_\delta(x_1,x_2,x_3) \in A_\alpha$ if and only if $\langle \omega_i , \alpha \rangle = 0$ for $i=1,2,3$. But
\[
\langle \omega_i , \alpha \rangle = \langle \theta_i , \alpha \rangle - \frac{1}{2}\epsilon x_i \langle \delta , \alpha \rangle.
\]
Let $v,w \in \mathbb{R}^3$ be given by $v = ( \langle \theta_1 , \alpha \rangle , \langle \theta_2 , \alpha \rangle , \langle \theta_3 , \alpha \rangle )$ and $w = \frac{1}{2}\epsilon \langle \delta , \alpha \rangle ( x_1  , x_2 , x_3)$. Then $f_\delta(x_1,x_2,x_3) \in A_\alpha$ if and only if $v = w$. Let $|| \; ||$ denote the standard norm on $\mathbb{R}^3$. Then if $||w|| < ||v||$, it follows that $f_\delta(x_1,x_2,x_3) \notin A_\alpha$. But $||v||^2 = \langle \theta_1 , \alpha \rangle^2 + \langle \theta_2 , \alpha \rangle^2 + \langle \theta_3 , \alpha \rangle^2 = \alpha_H^2$ and 
\[
||w||^2 = \frac{1}{4} \epsilon^2 \langle \delta , \alpha \rangle^2 (x_1^2 + x_2^3 + x_3^2) = \frac{1}{4} \epsilon^2 \langle \delta , \alpha \rangle^2.
\]
Recall that $\delta \in H^\perp$, hence $\langle \alpha , \delta \rangle = \langle \alpha_K , \delta \rangle$. Now since $\alpha_K , \delta$ lie in the negative definite space $H^\perp$, we can apply Cauchy--Schwarz to $-\langle \; , \; \rangle$ on $H^\perp$ to deduce that
\[
\langle \alpha , \delta \rangle^2 = \langle \alpha_K , \delta \rangle^2 \le \sqrt{ |\alpha_K^2| | \delta^2 |} = \sqrt{ 2|\alpha_K^2|}.
\]
Then since $\alpha^2 = -2 = \alpha_H^2 - |\alpha_K^2|$, we get 
\[
\langle \alpha , \delta \rangle^2 \le \sqrt{ 2(\alpha_H^2+2)}
\]
and thus
\[
||w||^2 \le \frac{1}{4}\epsilon^2 \sqrt{ 2(\alpha_H^2+2)}.
\]
Hence if $\epsilon^2 < 4 (\alpha_H)^2/\sqrt{2(\alpha_H^2+2)}$, then $||w||^2 < ||v||^2$ and $f_\delta(x_1,x_2,x_3) \notin A_\alpha$.

Let $g(t) = 4t/\sqrt{2(t+2)}$. This is an increasing function on $[0,\infty)$. Suppose $\epsilon^2 < g(1) = 4/\sqrt{6}$. Then $f_\delta(x_1,x_2,x_3) \notin A_\alpha$ for all $\alpha$ such that $\alpha_H^2 \ge 1$. On the other hand, we claim that there are only finitely many $\alpha \in \Delta^+ \setminus \{\delta\}$ with $\alpha_H^2 < 1$. To see this, note that $-2 = \alpha^2 = \alpha_H^2 - |\alpha_K^2|$, so $|\alpha_K^2| = \alpha_H^2 +2 < 3$ and so $\alpha_H^2 + |\alpha_K^2| < 4$. But the map $x \mapsto ||x|| =  x_H^2 + |x_K^2|$ defines a norm on $\mathbb{R}^{3,19}$. So the set $\{ \alpha \in \Delta^+ \setminus \{\delta \} \; | \; \alpha_H^2 \le 1 \}$ is closed and bounded, hence compact. It is also discrete, hence finite. Thus by choosing $\epsilon$ such that $\epsilon^2 < 4/\sqrt{6}$ and such that $\epsilon^2 < g( \alpha_H^2)$ for the finitely many $\alpha \in \Delta^+ \setminus \{\delta\}$ with $\alpha_H^2 < 1$, we have that $f_\delta(x_1,x_2,x_3)$ does not lie on $A_{\delta'}$ for any $\delta' \in \Delta^+ \setminus \{\delta\}$.
\end{proof}

By the Lemma, we may choose an $\epsilon$ such that $f_\delta(x_1,x_2,x_3)$ does not lie on $A_{\delta'}$ for any $\delta' \in \Delta^+$ not equal to $\delta$. Moreover, we have
\[
( \langle \omega_1 , \delta \rangle , \langle \omega_2 , \delta \rangle , \langle \omega_3 , \delta \rangle ) = \epsilon(x_1,x_2,x_3).
\]
Then since $(x_1,x_2,x_3) \in S^2$, we see that at least one of $\langle \omega_1$, $\delta \rangle , \langle \omega_2 , \delta \rangle$, and $\langle \omega_3 , \delta \rangle$ must be non-zero. So $f(x_1,x_2,x_3)$ does not lie on $A_\delta$. It follows that $f$ maps to $W$. So there is a map $g_\delta : B \to T_{Ein}$ such that $f_\delta = P \circ g_\delta$, namely $g_\delta = P^{-1} \circ f_\delta$.

\begin{definition}
Let $E_\delta = g_\delta^*( E_{Ein})$ be the pullback of the family $E_{Ein} \to T_{Ein}$ by the map $g_\delta : S^2 \to T_{Ein}$. This is a smooth family of $K3$ surfaces over $S^2$.
\end{definition}

\begin{remark}
In addition to $\delta$, the map $g_\delta : S^2 \to T_{Ein}$ depends on a choice of $p \in A_\delta \setminus (A_\delta \cap ( \cup_{\delta' \in \Delta^+ \setminus \{\delta\} } A_{\delta'})$, a sufficiently small $\epsilon > 0$ and a choice of oriented orthonormal basis $(\{\theta_1 , \theta_2 , \theta_3\})$. We will show this space is path-connected. Since moving $(p , \epsilon , \{\theta_1 , \theta_2 , \theta_3\} )$ along a path only changes $g_\delta$ by isotopy, this will imply that that family $E_\delta$ is well defined up to isomorphism. To see connectedness, first note that $A_\delta \setminus (A_\delta \cap ( \cup_{\delta' \in \Delta^+ \setminus \{\delta\} } A_{\delta'})$ is path-connected. Any two points $p_0, p_1$ can be joined by a smooth path $p_t$ valued in $A_\delta$. By \cite[Theorem 2.5]{hir} we can assme $p_t$ is transverse to the countably many codimension $3$ submanifolds $\{ A_\delta \cap A_{\delta'} \}_{\delta' \in \Delta^+ \setminus \{\delta\}}$ in $A_\delta$. Since $p_t$ is $1$-dimensional, transversality means $p_t$ is disjoint from each $A_\delta \cap A_{\delta'}$ and hence $p_t$ is a path in $p \in A_\delta \setminus (A_\delta \cap ( \cup_{\delta' \in \Delta^+ \setminus \{\delta\} } A_{\delta'})$. Next, using compactness of $[0,1]$ we can find an $\epsilon$ such that $(p_t , \epsilon)$ is in the space of pairs for all $t$. Lastly, the space of all triples $(p , \epsilon , \{ \theta_1 , \theta_2 , \theta_3 \} )$ is a principal $SO(3)$-bundle over the space of pairs $(p , \epsilon)$, so it is also path-connected.
\end{remark}

\begin{remark}
By construction, the family $E_{Ein} \to T_{Ein}$ has structure group ${\rm TDiff}(X)$, hence the same is true of the pullback family $E_\delta$. To say that $E_\delta$ has structure group ${\rm TDiff}(X)$ amounts to saying that $E_\delta$ is equipped with a trivialisation $\mathcal{H}^2(X) \cong H^2(X ; \mathbb{R}) \times B$ of the local system $\mathcal{H}^2(X)$.
\end{remark}

For each $\delta \in \Delta^+$ we have constructed a homotopy class of map $g_\delta : S^2 \to T_{Ein}$. As $T_{Ein}$ is simply connected \cite[Page 4]{bako3}, there is a bijection between unbased homotopy classes of maps $S^2 \to T_{Ein}$ and the homotopy group $\pi_2( T_{Ein})$. Hence $g_\delta$ defines a class $[g_\delta] \in \pi_2(T_{Ein})$. Now since $T_{Ein} = Ein/{\rm TDiff}(X)$, the long exact sequence of homotopy groups gives a map
\[
\partial : \pi_2(T_{Ein}) \to \pi_1( {\rm TDiff}_0(X) ) = \pi_1({\rm Diff}_0(X))
\]
where the second equality is by Equation (\ref{equ:tdiff0}). In particular, we may define $h_\delta = \partial [ g_\delta ] \in \pi_1({\rm Diff}_0(X))$. Applying the clutching construction to $h_\delta$, we recover the family $E_\delta$.

Since $T_{Ein}$ is simply connected, the Hurwitz theorem gives an isomorphism $\pi_2( T_{Ein} ) = H_2( T_{Ein} ; \mathbb{Z})$. From \cite[Lemma 5.3]{gia}, we have
\[
H_2( T_{Ein} ; \mathbb{Z}) = \bigoplus_{\delta \in \Delta^+} \mathbb{Z}[g_\delta].
\]
Hence $\pi_2( T_{Ein} )$ is a free abelian group with generators the maps $\{ [g_\delta] \}_{\delta \in \Delta^+}$.

Recall that a $K3$ surface satisfies 
\[
b^+(X) = 3, \; \sigma(X) = -16, \; b_1(X) = 0.
\]
Since $X$ is spin and simply-connected, for each $u \in H^2( X ; \mathbb{Z})$, there is a uniquely determined spin$^c$-structure $\mathfrak{s}_u$ for which $c_1(\mathfrak{s}_u) = 2u$. Then
\[
d(X , \mathfrak{s}_u) = \frac{ (2u)^2 + 16 }{4} - 1 - 3 = u^2.
\]

Let $\alpha \in \Delta$. Then $\alpha^2 = -2$ and hence $d(X , \mathfrak{s}_\alpha ) = -2$. Therefore we have the Seiberg--Witten invariant
\[
sw_{\mathfrak{s}_\alpha} : \pi_1( {\rm Diff}_0(X) ) \to \mathbb{Z}.
\]
To simplify notation we will write $sw_{\alpha}$ for $sw_{\mathfrak{s}_\alpha}$. If $\alpha \in \Delta$, then from Proposition \ref{prop:conj}, we have
\[
sw_{-\alpha} = -sw_{\alpha}.
\]
For this reason, it suffices to only consider the homomorphisms $sw_\alpha$ for $\alpha \in \Delta^+$.

From Theorem \ref{thm:finite}, we have that for each $f \in \pi_1({\rm Diff}(X))$, $sw_\alpha(f)$ is non-zero for only finitely many $\alpha \in \Delta^+$. Therefore, we obtain a homomorphism
\[
sw : \pi_1({\rm Diff}_0(X)) \to \bigoplus_{\alpha \in \Delta^+} \mathbb{Z}, \quad sw(f) = \oplus_{\alpha} sw_\alpha(f).
\]

Recall that we constructed classes $h_\delta = \partial [g_\delta] \in \pi_1({\rm Diff}_0(X))$ such that the family $E_\delta$ is obtained from the clutching construction applied to $h_\delta$.

\begin{theorem}\label{thm:swinv}
Let $\alpha, \delta \in \Delta^+$. Then
\[
sw_\alpha( h_\delta ) = \begin{cases} 1 & \text{if } \alpha = \delta, \\ 0 & \text{otherwise}. \end{cases}
\]
\end{theorem}
\begin{proof}
By definition, $sw_\alpha( h_\alpha )$ is the Seiberg--Witten invariant of $(E_\delta ,  \mathfrak{s}_\alpha)$ with respect to the canonical chamber, where $E_\delta$ is the family obtained by the clutching construction applied to $h_\delta$. 

We recall the construction of $E_\delta$. Choose a point $p \in A_\delta$ such that $p$ does not lie on any $A_{\delta'}$ for $\delta' \in \Delta^+$ other than $\delta$. Let $H \subset H^2(X ; \mathbb{R})$ be the positive definite $3$-plane corresponding to $p$. Choose a basis $\theta_1, \theta_2, \theta_3$ for $H$ satisfying $\langle \theta_i , \theta_j \rangle = \delta_{ij}$. Since $p \in A_\delta$, we have $\langle \theta_j , \delta \rangle = 0$ for $j = 1,2,3$. Let $B = S^2$ be the unit sphere in $\mathbb{R}^3$. We take $E_\delta \to S^2$ to be the pullback of the family $E_{Ein} \to T_{Ein}$ by a map $g_\delta : S^2 \to T_{Ein}$. Let $P : T_{Ein} \to Gr_3( \mathbb{R}^{3,19})$ be the period map. Then $f_\delta = P \circ g_\delta : S^2 \to Gr_3( \mathbb{R}^{3,19})$ is defined as
\[
f_\delta(x_1,x_2,x_3) = {\rm span}( \omega_1 , \omega_2 , \omega_3 ),
\]
where $\epsilon > 0$ is sufficiently small and
\[
\omega_i = \theta_i - \epsilon x_i \delta/2, \text{ for } i = 1,2,3.
\]
Let $\rho_H : H^2(X ;  \mathbb{R}) \to H$ be the projection to $H$ with kernel $H^\perp$. Then since $\theta_1 , \theta_2, \theta_3 \in H$ and $\delta \in H^\perp$, we have
\[
\rho_H( \theta_i ) = \theta_i, \quad \rho_H( \delta ) = 0.
\]
In particular, this gives
\[
\rho_H( \omega_i ) = \theta_i.
\]
Let $\mathcal{H}^+_g(X) \to S^2$ be the bundle of harmonic self-dual $2$-forms of the family $E_\delta$. Then $\omega_1 , \omega_2 , \omega_3$ is a frame for $\mathcal{H}^+_g(X)$. Let $\varphi : \mathcal{H}^+_g(X) \to H$ be the inclusion into $H^2(X ; \mathbb{R})$ followed by $\rho_H$. Define $w(\alpha)$ to be the section of $H$ given by 
\[
w(\alpha)(b) = \varphi( 2 \pi c_1(\mathfrak{s}_{\alpha})^{+_{g_b}} ) = \varphi( 4 \pi \alpha^{+_{g_b}} ).
\]

Let $\omega_1^*, \omega_2^*, \omega_3^*$ be the dual frame of $\mathcal{H}^+_{g_b}$, defined by the condition
\[
\langle \omega_i , \omega_j^* \rangle = \delta_{ij},
\]
where $\langle \; , \; \rangle$ is the intersection pairing on $H^2(X ; \mathbb{R})$. From $\omega_i = \theta_i - \epsilon x_i \delta / 2$, one finds
\[
\langle \omega_i , \omega_j \rangle = \delta_{ij} - \epsilon^2 x_i x_j/2.
\]
One can then directly check that the dual frame is given by
\begin{equation}\label{equ:omegastar}
\omega^*_i = \omega_i + \mu x_i ( x_1 \omega_1 + x_2 \omega_2 + x_3 \omega_3),
\end{equation}
where 
\[
\mu = \frac{ \epsilon^2/2}{1 - \epsilon^2/2}.
\]
We have
\[
\alpha^{+_g} = \langle \alpha , \omega^*_1 \rangle \omega_1 + \langle \alpha , \omega^*_2 \rangle \omega_2 + \langle \alpha , \omega^*_3 \rangle \omega_3.
\]
Applying $\rho_H$, we get
\[
w(\alpha) = 4\pi \left( \langle \alpha , \omega^*_1 \rangle \theta_1 + \langle \alpha , \omega^*_2 \rangle \theta_2 + \langle \alpha , \omega^*_3 \rangle \theta_3 \right).
\] 
We use the basis $\theta_1, \theta_2, \theta_3$ to identify $H$ with $\mathbb{R}^3$. Then
\[
w(\alpha) = 4\pi \left( \langle \alpha , \omega^*_1 \rangle , \langle \alpha , \omega^*_2 \rangle , \langle \alpha , \omega^*_3 \rangle \right).
\]

Suppose that $\alpha \neq \delta$. Then since $p$ does not lie on $A_{\alpha}$, we have $\langle \theta_j , \alpha \rangle \neq 0$ for some $j$. From (\ref{equ:omegastar}), we get
\[
w(\alpha) = 4\pi \left( \langle \alpha , \theta_1 \rangle  , \langle \alpha , \theta_2 \rangle , \langle \alpha , \theta_3 \rangle \right) + O(\epsilon)
\]
where $O(\epsilon)$ denotes terms of order $\epsilon$. Then (for sufficiently small $\epsilon$) it follows that $w(\alpha)$ is non-vanishing and that ${\rm deg}( w(\alpha) ) = 0$. More precisely, since $w$ is non-vanishing, it defines a map $w/||w| : S^3 \to S(H)$ where $S(H)$ is the unit sphere in $H$ and this map has degree $0$. Here we give $S(H)$ the induced orientation (recall that $H$ has a canonical orientation). The map has degree $0$ because there is a homotopy from $w$ to a constant map, given by contracting the $O(\epsilon)$ term to zero.

Suppose instead that $\alpha = \delta$. In this case we find
\[
w(\alpha) = w(\delta) = 4\pi \epsilon ( x_1 , x_2 , x_3 ) + O(\epsilon^2).
\]
So (for sufficiently small $\epsilon$) $w(\alpha)$ is non-vanishing and ${\rm deg}( w( \alpha ) ) = 1$ (since the map $(x_1,x_2,x_3) \mapsto 4\pi \epsilon (x_1,x_2,x_3)$ has degree $1$). In all cases, we see that $w(\alpha)$ is non-vanishing.

The family $E_{Ein}$ has a fibrewise metric $g_{Ein}$ which is a Ricci flat Einstein metric on each fibre. Let $g = g_\delta^*( g_{Ein} )$ be the fibrewise metric on $E_\delta$ obtained by pullback. Then for each $b \in S^2$, the metric $g_b$ is Ricci flat and in particular has zero scalar curvature. Now consider the families Seiberg--Witten moduli space $\mathcal{M}( E_\delta , \mathfrak{s}_\alpha , g , 0)$ for the zero perturbation $\eta = 0$.

Suppose that $(A , \psi) \in \mathcal{M}( E_\delta , \mathfrak{s}_\alpha , g , 0)$ is a solution to the Seiberg--Witten equations in the family. Then $(A , \psi)$ is the solution of the Seiberg--Witten equations on some fibre $X_b$ with metric $g_b$ and zero perturbation. The Weitzenb\"ock formula together with the Seiberg--Witten equations and the fact that $g_b$ has zero scalar curvature implies that $\psi = 0$ (\cite[Corollary 2.2.6]{nic}). So every solution in $\mathcal{M}( E_\delta , \mathfrak{s}_\alpha , g , 0)$ is reducible.

On the other hand, since $w(\alpha)$ is non-vanishing, the perturbation $\eta = 0$ does not lie on the wall, that is, there are no reducible solutions. So the moduli space $\mathcal{M}( E_\delta , \mathfrak{s}_\alpha , g , 0 )$ is empty. Recall that in the proof of Theorem \ref{thm:finite}, we used the wall crossing formula to deduce the identity
\[
\int_{S^{n+1}} SW(\mathfrak{s} , \eta) = sw_{\mathfrak{s}}(f) - {\rm deg}( \phi_{\mathfrak{s} , \eta} )
\]
(see Equation (\ref{equ:wcf})). Taking $\mathfrak{s} = \mathfrak{s}_{\alpha}$, $f = h_\delta$ and $\eta = 0$, we obtain
\[
\int_{S^{n+1}} SW( \mathfrak{s}_{\alpha} , 0 ) = sw_{\alpha}( h_\delta) - {\rm deg}( w(\alpha) ).
\]
But $\mathcal{M}( E_\delta , \mathfrak{s}_\alpha , g , 0 )$ is empty, so $SW( \mathfrak{s}_{\alpha} , 0 ) = 0$ and hence
\[
sw_{\alpha}( h_\delta ) = {\rm deg}( w(\alpha) ).
\]
Further, we have already shown that
\[
{\rm deg}( w(\alpha) ) = \begin{cases} 1 & \text{if } \alpha = \delta, \\ 0 & \text{otherwise}. \end{cases}
\]
\end{proof}

As an immediate consequence of Theorem, \ref{thm:swinv}, we have:

\begin{theorem}
The homomorphism
\[
sw : \pi_1({\rm Diff}_0(X)) \to \bigoplus_{\alpha \in \Delta^+} \mathbb{Z}
\]
is surjective, hence $\pi_1({\rm Diff}_0(X))$ contains $\bigoplus_{\alpha \in \Delta^+} \mathbb{Z}$ as a direct summand (recall that $\pi_1({\rm Diff}_0(X))$ is abelian).
\end{theorem}

\begin{theorem}
The boundary map
\[
\partial : \pi_2( T_{Ein} ) \to \pi_1( {\rm TDiff}_0(X) ) = \pi_1( {\rm Diff}_0(X) )
\]
induced by the fibration $Ein \to Ein/{\rm TDiff}(X) = T_{Ein}$ admits a left inverse, given by
\[
\pi_1( {\rm Diff}_0(X) ) \to \pi_2(T_{Ein}), \quad x \mapsto \bigoplus_{\alpha} sw_\alpha(x) [ g_\alpha ].
\]
\end{theorem}
\begin{proof}
Recall that
\[
\pi_2( T_{Ein}) \cong H_2( T_{Ein} ; \mathbb{Z}) \cong \bigoplus_{\alpha \in \Delta^+} \mathbb{Z} [g_\alpha]
\]
and that $\partial [g_\alpha] = h_\alpha$. Hence, if we define $t : \pi_1({\rm Diff}_0(X)) \to \pi_2(T_{Ein})$ to be given by $t(x) = \bigoplus_{\alpha \in \Delta^+} sw_\alpha(x) [g_\alpha]$. Using Theorem \ref{thm:swinv}, it follows that $t \circ \partial = id$, so that $t$ is a left inverse of $\partial$, as claimed.
\end{proof}

From the homeomorphism $P : T_{Ein} \to W = Gr_3(\mathbb{R}^{3,19}) \setminus \bigcup_{\delta \in \Delta} A_\delta$, we see that $T_{Ein}$ is connected. Then since $T_{Ein} = Ein/{\rm TDiff}(X)$, it follows that ${\rm TDiff}(X)$ acts transitively on the connected components of $Ein$ and that the components of $Ein$ are all homeomorphic to each other. Choose arbitrarily a basepoint $p \in Ein$. Since the components of $Ein$ are all homeomorphic, the isomorphism class of $\pi_1(Ein , p)$ does not depend on the choice of $p$ and we simply write $\pi_1(Ein)$. From the long exact sequence in homotopy groups associated to $Ein \to T_{Ein}$ we get an exact sequence
\[
\cdots \to \pi_2( T_{Ein} ) \buildrel \partial \over \longrightarrow \pi_1( {\rm TDiff}_0(X) ) \to \pi_1(Ein) \to \pi_1( T_{Ein}).
\]
We have also seen that $T_{Ein}$ is simply-connected and that $\partial$ admits a left inverse, so we obtain an isomophism
\[
\pi_1({\rm TDiff}_0(X)) = \pi_1( {\rm Diff}_0(X) ) \cong \left( \bigoplus_{\alpha \in \Delta^+} \mathbb{Z} \right) \oplus \pi_1(Ein),
\]
where the summand $\bigoplus_{\alpha \in \Delta^+} \mathbb{Z}$ is detected by the Seiberg--Witten invariants $sw_{\mathfrak{\beta}}$.\\

\begin{remark}
Smooth families over $S^2$ with fibres diffeomorphic to $X$ correspond, via the clutching construction, to elements of $\pi_1({\rm Diff}_0(X))$ considered modulo the conjugation action of ${\rm Diff}(X)$. For this reason we are interested in the action of ${\rm Diff}(X)$ on $\pi_1({\rm Diff}_0(X))$. The Seiberg--Witten invariants are compatible with this action in the following sense. Let $\{ e_\alpha \}_{\alpha \in \Delta^+}$ denote the standard basis for $\bigoplus_{\alpha \in \Delta^+} \mathbb{Z}$. For $f \in {\rm Diff}(X)$ and $x \in H^2(X ; \mathbb{R})$, let us write $f_*(x) = (f^{-1})^*(x)$ so that $(f , x) \mapsto f_*(x)$ is a left action. Let ${\rm Diff}(X)$ act on $\bigoplus_{\alpha \in \Delta^+} \mathbb{Z}$ by setting
\[
f \cdot e_\alpha = \begin{cases} e_{ f_* \alpha } & \text{if } f_*\alpha \in \Delta^+, \\ -e_{ -f_*\alpha } & \text{if } f_*\alpha \in \Delta^-. \end{cases}
\]
Then it follows easily from Propositions \ref{prop:diff} and \ref{prop:conj}  that the map $sw : \pi_1({\rm Diff}_0(X)) \to \bigoplus_{\alpha \in \Delta^+} \mathbb{Z}$ is ${\rm Diff}(X)$-equivariant.

\end{remark}


\bibliographystyle{amsplain}

\end{document}